\documentclass[a4paper,10pt]{amsart}
\def\sourcecode{}
\def\solution{}
\def\reviseone{}
\usepackage[all]{xy}
\usepackage{amsthm,amsmath,amssymb}
\usepackage{latexsym,mathrsfs,comment,float}
\usepackage{seqsplit}
\usepackage{mathscinet}
\usepackage{amsrefs}
\usepackage{verbatim}
\usepackage[autostyle]{csquotes}
\usepackage{xcolor}
\usepackage{pb-diagram,pb-xy}
\usepackage{listings}
\usepackage[notext]{stix2}
\def\mathbold#1{\ifdefined\stix@dotlessi\mathbb{#1}\else\text{\boldmath{$#1$}}\fi}
\title{Topological Complexity of $S^3/Q_8$ \\ as \\ fibrewise L-S category}
\date{\today}
\keywords{topological complexity, space form, quaternion group, python program}
\subjclass[2020] {Primary: 55M30; Secondary: 55R70, 55M35, 55P35, 57T30}
\author[Iwase]{Norio Iwase}
\address[Iwase]{Faculty of Mathematics, Kyushu University}
\email[Iwase]{iwase@math.kyushu-u.ac.jp}
\author[Miyata]{Yuya Miyata}
\address[Miyata]{Graduate School of Mathematics, Kyushu University}
\email[Miyata]{katorisenkouhanabi@gmail.com}
\theoremstyle{theorem}
\newtheorem{thm}{{\textbf{Theorem}}}
\newtheorem{lem}[thm]{{\textbf{Lemma}}}

\newtheorem{prop}[thm]{{\textbf{Proposition}}}
\newtheorem{fact}[thm]{{\textbf{Fact}}}
\theoremstyle{definition}
\newtheorem{definition}[thm]{{\textbf{Definition}}}
\newtheorem{notation}[thm]{{\textbf{Notation}}}

\theoremstyle{remark}
\newtheorem{remark}[thm]{{\textbf{Remark}}}
\numberwithin{equation}{section}
\numberwithin{thm}{section}
\usepackage{paralist}
\makeatletter
\newenvironment{enumerate*}[1]{\vspace{.5ex}
\begin{inparaenum}[\hspace{.4em}\begin{minipage}{1.8em}\hfill #1\end{minipage}]
}{\end{inparaenum}}
\newenvironment{subenumerate*}{%
 \let\par\relax\advance\@enumdepth\@ne%
 \setcounter{enumii}{1}%
\begin{inparaenum}[\hspace{.4em}\begin{minipage}{1.8em}\hfill i)\end{minipage}]
}{\setcounter{enumii}{0}\end{inparaenum}}%
\def\vitem{\endgraf\vskip1ex\noindent\item}
\def\hitem{\hskip.0em\item}

\makeatother
\def\textem#1{{\color{red}#1}}
\def\reduced#1{\bar{#1}}

\def\integral{\mathbb{Z}}

\def\quaternion{\mathbb{H}}
\def\field{\mathbb{F}}
\def\cupp{\text{\scriptsize$\,\bf\smile\,$}}

\renewcommand{\hom}{\operatorname{Hom}}

\makeatletter
\def\emptyarg{}
\def\@wgt#1{\def\thisarg{#1}\ifx\thisarg\emptyarg\operatorname{wgt}\else\operatorname{wgt}\hspace{.06em}(#1)\fi}
\def\@@wgt[#1]#2{\def\thisarg{#2}\ifx\thisarg\emptyarg\operatorname{wgt}_{#1}\else\operatorname{wgt}_{#1}(#2)\fi}
\def\wgt{\@ifnextchar[{\@@wgt}{\@wgt}}
\def\@Mwgt#1{\def\thisarg{#1}\ifx\thisarg\emptyarg\operatorname{Mwgt}\else\operatorname{Mwgt}\hspace{.06em}(#1)\fi}
\def\@@Mwgt[#1]#2{\def\thisarg{#2}\ifx\thisarg\emptyarg\operatorname{Mwgt}_{#1}\else\operatorname{Mwgt}_{#1}(#2)\fi}
\def\Mwgt{\@ifnextchar[{\@@Mwgt}{\@Mwgt}}
\def\@cupl#1{\def\thisarg{#1}\ifx\thisarg\emptyarg\operatorname{cup}\else\operatorname{cup}\hspace{.06em}(#1)\fi}
\def\@@cupl[#1]#2{\def\thisarg{#2}\ifx\thisarg\emptyarg\operatorname{cup}_{#1}\else\operatorname{cup}_{#1}(#2)\fi}
\def\cupl{\@ifnextchar[{\@@cupl}{\@cupl}}
\def\@catBb#1{\def\thisarg{#1}\ifx\thisarg\emptyarg\operatorname{cat^{*}_{B}}\else\operatorname{cat^{*}_{B}}\hspace{.06em}(#1)\fi}
\def\@@catBb(#1){\operatorname{cat^{*}_{B}}\hspace{.06em}(#1)}
\def\catBb{\@ifnextchar({\@@catBb}{\@catBb}}
\def\@cat#1{\def\thisarg{#1}\ifx\thisarg\emptyarg\operatorname{cat}\else\operatorname{cat}\hspace{.06em}(#1)\fi}
\def\@@cat[#1]#2{\def\thisarg{#2}\ifx\thisarg\emptyarg\operatorname{cat}_{#1}\else\operatorname{cat}_{#1}(#2)\fi}
\def\cat{\@ifnextchar[{\@@cat}{\@cat}}
\def\@tc#1{\def\thisarg{#1}\ifx\thisarg\emptyarg\operatorname{tc}\else\operatorname{tc}\hspace{.06em}(#1)\fi}
\def\@@tc{\operatorname{tc}}
\def\tc{\@ifnextchar({\@@tc}{\@tc}}
\def\@dim#1{\def\thisarg{#1}\ifx\thisarg\emptyarg\operatorname{dim}\else\operatorname{dim}\hspace{.06em}(#1)\fi}
\def\@@dim[#1]#2{\def\thisarg{#1}\ifx\thisarg\emptyarg\operatorname{dim}_{#1}\else\operatorname{dim}_{#1}(#2)\fi}
\def\dim{\@ifnextchar[{\@@dim}{\@dim}}
\def\@tcm#1{\def\thisarg{#1}\ifx\thisarg\emptyarg\operatorname{tc^{M}}\else\operatorname{tc^{M}}(#1)\fi}
\def\@@tcm{\operatorname{tc^{M}}}
\def\tcm{\@ifnextchar({\@@tcm}{\@tcm}}
\def\@TC#1{\def\thisarg{#1}\ifx\thisarg\emptyarg\operatorname{TC}\else\operatorname{TC}\hspace{.06em}(#1)\fi}
\def\@@TC{\operatorname{TC}}
\def\TC{\@ifnextchar({\@@TC}{\@TC}}
\def\@ad#1{\def\thisarg{#1}\ifx\thisarg\emptyarg\operatorname{ad}\else\operatorname{ad}\hspace{.06em}(#1)\fi}
\def\@@ad{\operatorname{ad}}
\def\ad{\@ifnextchar({\@@ad}{\@ad}}
\makeatother
\def\category#1{\mathcal{#1}}

\def\TopBB{\mathcal{NG}_{B}^{B}}
\def\path#1{\operatorname{\mathcal{P}}(#1)}

\def\Loop{\operatorname{\Omega}}
\def\LoopB#1{\operatorname{\Loop_{B}}(#1)}

\def\wgtB#1{\operatorname{wgt}_{B}(#1)}
\def\wgtBB#1{\operatorname{wgt}_{B}^{B}(#1)}

\def\catBB#1{\operatorname{cat}_{B}^{B}(#1)}
\def\cuplB#1{\operatorname{cup}_{B}(#1)}

\def\double#1{\operatorname{\it d}(#1)}

\def\id{\operatorname{id}}
\def\proj{\operatorname{pr}}

\def\map#1{\operatorname{Map}(#1)}

\def\midvert{ \ \mathstrut\vrule\hskip.1pt\ }
\def\Min{\operatorname{Min}}
\def\Max{\operatorname{Max}}
\def\Img{\operatorname{Im}}

\def\product{\textstyle\Pi}
\def\fatvee{\large\text{T}}
\def\homeo{\approx}
\def\comp{\hskip1pt\text{\footnotesize$\circ$}\hskip1pt}
\makeatletter
\def\claimname{Claim}
\def\preclaimword{\hskip1.68em}

\def\@postclaim[#1]{\bf #1}
\def\@Claim[#1]{\preclaimword {\bf #1.} \@ifnextchar[{\@postclaim}{}}
\newenvironment{Claim}%
{\par\vskip1.5ex\noindent\@ifnextchar[{\@Claim}{\bf\claimname~}}%
{\unskip\nobreak\hfill\par\vskip1.5ex}
\makeatother

\begin{document}
\begin{abstract}
In 2010, M.~Sakai and the first author showed that the topological complexity of a space $X$ coincides with the fibrewise unpointed L-S category of a pointed fibrewise space $\proj_{1} : X \times X \to X$ with the diagonal map $\Delta : X \to X \times X$ as its section.
In this paper, we describe our algorithm how to determine the fibrewise L-S category or the Topological Complexity of a topological spherical space form. Especially, for $S^3/Q_8$ where $Q_8$ is the quaternion group, we write a python code to realise the algorithm to determine its Topological Complexity.
\end{abstract}

\maketitle\baselineskip18pt

\section{Intoroduction}
Topological complexity was introduced in \cite{MR1957228} by Michael Farber as a numerical homotopy invariant.
It attracts many authors including people working on similar homotopy invariant of Lusternik-Schnirelmann category, L-S category, for short.
Recently, many authors started to use a `normalized' or a `reduced' version of it including Farber himself.
In this paper, we use the symbol `$\tc{}$' for the reduced version of it to distinguish from the original denoted by $\TC{}$.

The reduced version of topological complexity is defined as follows:
let $X$ be a path-connected space, $\path{X}=\map{[0,1],X}$ and $\varpi : \path{X} \to M \times M$ the projection given by $\varpi(u) = (u(0),u(1))$.
Using a projection $p_{t} : \path{X} \to X$ defined by $p_{t}(u)=u(t)$, $t \in \{0,1\}$, we may write $\varpi=p_{0} \times p_{1}$.
Topological complexity of $X$, denoted by $\tc{X}$, is the least integer $n$ such that there is an open covering $U_0, \dots, U_n$ of $M \times M$ on each of which $\varpi$ admits a section.
We remark that, if a subset $A \subset X$ is contractible in $X$, there is a section of $\varpi$ on $A$.
The definition reminds us a well-known homotopy invariant: L-S category of a space $X$, denoted by $\cat{X}$, is the least integer $n$ such that there is an open covering $U_0, \dots,U_n$ of $X$ each of which is contractible in $X$.

We can verify the following relation among these invariants (see Farber \cite{MR1957228}).
\[ \cat{X} \le \tc{X} \le \cat{X \times X} \le 2\cat{X}.\]
More practically for an abelian group $R$, Farber introduced the zero-divisors ideal $I_{\varpi}(X;R)=\ker \Delta^{*} : H^{*}(X \times X;R) \to H^{*}(X;R)$ and the zero-divisors cup-length $\cupl[\varpi]{X;R}$ for a space $X$ \cite{MR1957228}, and then the TC-weight $\wgt[\varpi]{z;R}$ for $z \in I_{\varpi}(X;R)$ with Mark Grant \cite{MR2407101}:
$$
\cupl[\varpi]{X;R} \le \Max\{\wgt[\varpi]{z;R} \mid z \in I_{\varpi}(X;R) \smallsetminus \{0\}\} \le \tc{X}.
$$
In this paper, we adopt fibrewise method and skip the precise definition of the above notions.

Since the theory of topological complexity is growing and spreading rapidly, there are many open problems.
Among such problems, we are especially interested in the topological complexity of spherical space forms, and in particular, real projective spaces, aiming to give a natural and computational way to determine topological complexity, since Farber showed that the topological complexity of a real projective space coincides with its immersion dimension.

In this paper, we focus on the work of Kenso Fujii on the $K$-theory of a spherical space forms obtained as the orbit space of the unit sphere of $\quaternion^{t}$, where $t$ is a non-negative integer and $\quaternion$ is the set of all quaternionic numbers, by the diagonal action of the subgroup $Q_8$ of $\mathrm{Sp}(1)$, represented as $Q_{8}=\langle a, b \midvert a^{4} \!= b^4 \!= aba\bar{b} = 1, \,b^2 \!= a^{2} \rangle$, where $\bar{g}$ stands for $g^{-1}$ for any element $g \!\in\! Q_{8}$.

Our results are stated in \S \ref{sect:results} and the proofs are in \S \ref{sect:mainprop}-\ref{sect:mainlem}, but the actual calculation is done with computer.
Because of the size of the computer resources required, we gave up to use the usual method for topological complexity but the fibrewise L-S method by Iwase-Sakai \cite{MR2556074} and Iwase-Sakai-Tsutaya \cite{MR3975098}.
It results that we don't need to calculate the bar resolution of $G \times G$ but $G$ fibrewise.
This significantly reduces the computer resources required, while the explicit answer to our equation is still too long to print out.
\ifdefined\sourcecode
So we include the algorithm and python code in Appendix, instead.
\else
So we include the algorithm of the program in Appendix, instead.
\fi
We hope our method can be applied to more general cases. 

We work in the convenient category $\category{NG}$, introduced by Shimakawa-Yoshida-Haraguchi \cite{MR3884529}, and every chain groups are assumed to be $\field_{2}$-modules in this paper, unless otherwise stated.

\section{Results}\label{sect:results} 	

In this section, $R$ is assumed to be an $\field_{2}$-module.
Let us recall the James fibrewise theory.

\begin{definition}[\cite{MR1361889}]
Fibrewise spaces and maps are defined as follows.
\begin{enumerate}
\item 
A fibrewise space is a tuple $(E,p,X)$ consisting of spaces $E$ and $X$ with a map $p : E \to X$ called a projection.
A fibrewise map from a fibrewise space $(F,q,Y)$ to a fibrewise space $(E,p,X)$ is a pair $(\phi,f)$ of maps $\phi : F \to E$ and $f : Y \to X$ satisfying $p \comp \phi = f \comp q$ as $(\phi,f) : (F,q,Y) \to (E,p,X)$.
$(E,p,X)$ is often denoted simply by $E$, and $(\phi,f)$ by $\phi$.
\item 
A fibrewise pointed space is a pair $(E,s)$ of a fibrewise space $E=(E,p,X)$ and a section $s : X \to E$ of $p$, i.e, $p \comp s = \id_X$.
A fibrewise pointed map from a fibrewise pointed space $(F,r)$ to a fibrewise pointed space $(E,s)$ is a fibrewise map $(\phi,f) : (F,r) \to (E,s)$ satisfying $\phi \comp r = s \comp f$. 
$(E,s)$ is often denoted simply by $E$, and $(\phi,f)$ by $\phi$.
\end{enumerate}
\end{definition}
For instance, we have a Borel construction $EG \times_{G} X$ over $EG/G$ as a fibrewise space for a topological group $G$ and a $G$-space $X$, where we denote by $EG$ some contractible free $G$-space.
Further, for an adjoint action of $G$ on itself, we obtain the Borel construction denoted by $EG\times_{\ad{}}G$ as a fibrewise pointed space over $EG/G$, which is, in fact, a fibrewise group.

James has also introduced a fibrewise version of an ordinary cohomology as a direct summand of the oridinary cohomology of the total space.
From now on, we use a subscript $B$ to indicate that the notion is `fibrewise notion' over some base space, even if the base space is not the same as $B$. Similarly, we use a superscript $B$ to indicate that the notion is `fibrewise pointed'.

For a fibrewise pointed space $(E,p,X,s)$, the base space $X$ is a retract of $E$, and hence $H^{*}(E,s(X);R) \cong \ker\,s^{*} : H^{*}(E;R) \to H^{*}(X;R)$ can be regarded as a direct summad of $H^{*}(E;R)$.
\begin{definition}[\cite{MR1361889}]
For a fibrewise pointed space $E=(E,p,X,s)$, $H_{B}^{\ast}(E;R)=H^{*}(E,s(X);R)$ is called its fibrewise pointed cohomology.
Then a fibrewise map $\phi : (F,q,Y) \to (E,p,X)$ induces a homomorphism $\phi^{*} : H_{B}^{\ast}(E;R) \subset H^{\ast}(E;R) \to H^{\ast}(F;R)$, while a fibrewise pointed map $\phi : (F,q,Y,r) \to (E,p,X,s)$ induces a homomorphism $\phi^{*} : H_{B}^{\ast}(E;R) \to H_{B}^{\ast}(F;R) \subset H^{\ast}(F;R)$.
\end{definition}

For any map $f : Y \to X$, we have a fibrewise pointed space $E_{f}=(Y \times X,\ifdefined\reviseone{\proj_{1}}\else{\proj_{2}}\fi,Y,(\id_{X} {\times} f){\comp}\Delta)$, where $\proj_{t} : X_{1} \times X_{2} \to X_{t}$ denotes the canonical projection to the $t$-th factor, $t \!=\! 1, \,2$, and $\Delta : X \to X \times X$ denotes the diagonal.
Based on the ideas introduced by James \cite{MR1361889}, Iwase-Sakai \cite{MR2556074} introduced a fibrewise version of (unpointed) L-S category as follows. 

\begin{definition}[\cite{MR2556074}, Definition 1.6]\label{defn:fibrewiseLS}
Let $E=(E,p,X,s)$ be a fibrewise pointed space and $\phi=(\phi,f) : F=(F,q,Y) \to E=(E,p,X)$ be a fibrewise map.
Then the fibrewise (\textit{unpointed}) L-S category of $\phi$, denoted by $\catBb{\phi}$, is the least integer $k \ge 0$ such that there is a cover of $F$ by $(k{+}1)$ open subsets $\{U_i\}$ on each of which $\phi|U_{i}$ is fibrewise homotopic to $s \comp f \!\comp q|_{U_i}$. Then we define a fibrewise (\textit{unpointed}) L-S category of $E$ by $\catBb{E}=\catBb{\id_{E}}$. 
\end{definition}
\begin{remark}
If we stay in the fibrewise pointed category, we obtain James original fibrewise L-S category $\catBB{E}$ of a fiberwise well-pointed space $E$, using fibrewise pointed spaces and maps.
\end{remark}

Firstly, for an extensive use of homotopy theory, we alter the definition of L-S category, following George W. Whitehead by replacing an open cover $\{U_{i}\}$ of a space $X$ with a closed cofibration $F_{i} \hookrightarrow X$  covering $X$.
Then we say $\cat{X} \le t$ if the $t{+}1$-fold diagonal $\Delta^{t+1} : X \to \overset{t+1}\product X$ is compressible into the fat wedge $\overset{t+1}\fatvee X$, where $\overset{t+1}\fatvee X$ is defined by induction on $k \!\ge\! 1$ by $(\overset{1}{\product}X,\overset{1}{\fatvee}X) = (X,\ast)$ and 
\begin{align*}&
(\overset{k+1}{\product}X,\overset{k+1}{\fatvee}X) = (\overset{k}{\product}X \times X,\overset{k}{\fatvee}X \times X \cup \overset{k}{\product}X \times\ast).
\end{align*}
Secondly, we also alter the definition of fibrewise L-S category by replacing an open cover $\{U_{i}\}$ of $E$ with a closed cofibration $F_{i} \hookrightarrow E$ covering $E$ for a fibrewise pointed space $E=(E,p,X,s)$.
Then we say $\catBb{E} \le t$ if the $t{+}1$-fold fibrewise diagonal $\Delta_{B}^{t+1} : E \to \overset{t+1}{\product_{B}} E$ is compressible into the fibrewise fat wedge $\overset{t+1}{\fatvee_{\!B}} E$, where $\overset{t+1}{\fatvee_{\!B}} E$ is defined by induction on $k \!\ge\! 1$ by $(\overset{1}{\product_{B}}E,\overset{1}{\fatvee_{\!B}}E) = (E,s(X))$ and 
\begin{align*}&
(\overset{k+1}{\product_{B}}E,\overset{k+1}{\fatvee_{\!B}}E) = (\overset{k}{\product_{B}}E \!\times_{B}\! E, \overset{k}{\fatvee_{\!B}}E \!\times_{B}\! E \cup \overset{k}{\product_{B}}E \!\times_{B}\! s(X)).
\end{align*}
\begin{remark}
If we consider a monoidal motion planning, we must choose a fibrewise homotopy $h_{i}$ to keep the diagonal part fixed and thus, we obtain a fibrewise pointed map from $E$ to $\overset{t+1}{\fatvee_{\!B}} E \subset \overset{t+1}{\product_{B}} E$ as a fibrewise pointed compression of the fibrewise diagonal $\Delta_{B}^{t+1} : E \to \overset{t+1}{\product_{B}} E$.
\end{remark}

\begin{definition}[\cite{MR2556074}, Definition 6.3]
For a fibrewise pointed space $(E,p,X,s)$, we have the fibrewise pointed loop space $\LoopB{E}=(\LoopB{E},{p},X,{s})$ as follows:
\begin{align*}
\LoopB{E} &
= \{(b,\ell) \!\in\! X \times \path{E} \mid p \comp \ell \!=\! c(b), \ \ell(0) \!=\! \ell(1) \!=\! s (b) \},
\\&
{p}=\proj_{1}|\LoopB{E} : \LoopB{E} \subset X \times \path{E} \xrightarrow{\proj_{1}} X,\qquad 
\\&
{s}(b)=(b,c \comp s(b)),\quad
c(b)=\text{(the constant path at $b$)},
\end{align*}
together with an $A_{\infty}$-structure for $\LoopB{E}$ defined in $\TopBB$ as follows:
\begin{enumerate}
\item
$E_{B}^{t+1}(\LoopB{E})$ is the fibrewise homotopy pull-back of $X \hookrightarrow \overset{t+1}{\product_{B}}E \hookleftarrow \overset{t+1}{\fatvee_{\!B}}E$.
\vspace{.5ex}\item
$P_{\!B}^{t}(\LoopB{E}) \!=\! (P_{\!B}^{t}(\LoopB{E}),p_{t+1},X,s_{t+1})$ is the fibrewise homotopy pull-back of $E \xrightarrow{\Delta_{B}^{t+1}} \overset{t+1}{\product_{B}}E \hookleftarrow \overset{t+1}{\fatvee_{\!B}}E$, where $p_{t+1}^{-1}(b) = P^{t}\Loop{(p^{-1}(b))}, b \in X$.
\vspace{0ex}\item
$e^{E}_{t} : P_{\!B}^{t}(\LoopB{E}) \to E$\vspace{-.5ex} is induced from the inclusion $\overset{t+1}{\fatvee_{\!B}}E \hookrightarrow \overset{t+1}{\product_{B}}E$ by the fibrewise diagonal $\Delta_{B}^{t+1} : E \to \overset{t+1}{\product_{B}}E$, as an extension of $e_{t} : P^{t}\Loop{F} \to F$, where $F$ is the fibre of $p$.
\vspace{.5ex}\item
$p_{t+1}^{\LoopB{E}} : E_{B}^{t+1}(\LoopB{E}) \to P_{\!B}^{t}(\LoopB{E})$ is induced from $s : X \to E$.
\end{enumerate}
We remark that the section of $\overset{t+1}{\product_{B}}E$ is given by $\Delta_{B}^{t+1}{\comp}s : X \xrightarrow{s} E \xrightarrow{\Delta_{B}^{t+1}} \overset{t+1}{\product_{B}}E$.
\end{definition}

When the base point or a section for a fibrewise space is a closed cofibration, Iwase and Sakai showed that a fibrewise L-S category can be characterized in terms of fibrewise $A_{\infty}$-structure.
\begin{fact}\cite[Theorem 7.2]{MR2556074}\label{inj}
For a fibrewise pointed space $E$, we denote by $e_{t}^{E}$ the composition $P^{t}_{\!B}(\LoopB{E}) \hookrightarrow P^{\infty}_{\!B}(\LoopB{E}) \underset{\simeq}{\xrightarrow{e_{\infty}^{E}}} E$ for $t\!\ge\!0$. Then we have
\begin{center}
$\catBb{E} \le t$ $\iff$ $e_{t}^{E}$ has a right homotopy inverse.
\end{center}
\end{fact}
It enables us to define a stronger homotopy invariant, module weight.
\begin{definition}[Iwase-Sakai \cite{MR2556074}, Definition 8.3]
For $u \!\in\! H^{*}_{\!B}(E;R) \!\subset\! H^{*}(E;R)$, we define
\begin{align*}&
\cupl[B]{E;R} = \Max\{\,t\!\ge\! 0 \midvert \exists\,\{u_{1},\ldots,u_{t} \!\in\! H_{B}^{*}(E;R)\}\ \ u_{1}{\cdot}\cdots{\cdot}u_{t}\!\neq\!0\,\},
\\&
\wgt[B]{u;R} = \Max\{\,t \!\ge\! 0 \midvert \forall\,{\phi : F \!\to\! E} \ \catBb{\phi}\!<\!t \Rightarrow \phi^{*}(u)\!=\!0\,\},
\\&
\wgt[B]{E;R} = \Min\left\{\,t \!\ge\! 0 \midvert (e_{t}^{E})^{\ast}  \ \text{is monic}\,\right\},
\\&
\Mwgt[B]{E;\Gamma} = \Min\left\{\,t \!\ge\! 0 \midvert \text{\begin{minipage}{52mm}$\Img(e_{t}^{E})^{\ast}$ is a direct summand of $H^{*}(P^{t}_{\!B}(\LoopB{E});\field_{2})$ as a $\Gamma$-module\end{minipage}}\,\right\},
\end{align*}
where $\Gamma$ is an $\field_{2}$-subalgebra of $\mathcal{A}_{2}$ the modulo $2$ Steenrod algebra.
\end{definition}

Now, let us clarify the relationship between the above invariants and fibrewise L-S category. In \cite{MR2556074}, $\wgt[B]{u;R}$ is defined with $\wgtBB{u;R}$ which is introduced to give a lower bound for $\catBB{E}$ the original James fibrewise L-S category.
Since we do not know the equality of $\catBb{E}$ and $\catBB{E}$ as well as the equality of $\tcm{X}$ and $\tc{X}$ until now, let us state the following.

\begin{prop}\label{prop:multiplicativeformula}
For a fibrewise pointed space $(E,p,X,s)$, we obtain the following.
\begin{enumerate}
\item\label{prop:multiplicativeformula-1}
$\wgt[B]{u {\cdot} v;R} \ge \wgt[B]{u;R} + \wgt[B]{v;R}$
for $u, v \!\in\! H^{*}(E,X;R) \smallsetminus \{0\}$,
\item\label{prop:multiplicativeformula-2}
$\wgt[B]{u;R} = \Max\{ t \!\ge\! 0 \midvert (e_{t-1}^{E})^{\ast}(u) \!=\! 0\}$
for $u \!\in\! H^{*}(E,X;R) \smallsetminus \{0\}$,
\item\label{prop:multiplicativeformula-3}
$\wgt[B]{E;R}=\Max\{ \wgt[B]{u} \midvert u \!\in\! H_{B}^{\ast}(E;R) \smallsetminus \{0\}\}$,
\end{enumerate}
\end{prop}
\begin{proof}
(\ref{prop:multiplicativeformula-1}) \ 
Let us assume that $\wgt[B]{u;R}=m$ and $\wgt[B]{v;R}=n$, and that $\phi : (F,q,Y) \to (E,p,X)$ be a fibrewise map with $\catBb{\phi}<m{+}n$.
Then there is a cover of $F$ by $m{+}n$ open subsets $\{U_{i}\}$, at most, each on which $\phi|U_{i}$ is fibrewise homotopic to $s \comp f \!\comp q |U_{i}$.
Let $U = U_{1} \cup \cdots \cup U_{m}$ and $V = U_{m+1} \cdots \cup U_{m+n}$ to satisfy $\catBb{\phi|{U}}<m$ and $\catBb{\phi|{V}}<n$.
Hence $\phi^{*}(u)|_{U}=(\phi|{U})^{*}(u)=0$ in $H_{B}^{\ast}(U;R)$ and $\phi^{*}(v)|_{V}=(\phi|{V})^{*}(v)=0$ in $H_{B}^{\ast}(V;R)$.
Then we obtain $\phi^{*}(u{\cdot}v)=\phi^{*}(u){\cdot}\phi^{*}(v)=0$ by the definition of cup-products, which implies $\wgtB{u{\cdot}v;R} \ge m+n$.
\par(\ref{prop:multiplicativeformula-2}) \ 
Let $\wgt[B]{u;R}=m$. We can easily see that the filtration $X=P_{\!B}^{0}(\LoopB{E}) \subset P_{\!B}^{1}(\LoopB{E}) \subset \cdots \subset P_{\!B}^{t}(\LoopB{E})$ gives a fibrewise version of a cone decomposition of a fibrewise space $P_{\!B}^{t}(\LoopB{E})$, $t \!\ge\! 1$. It implies that $\catBb{P_{\!B}^{t}(\LoopB{E})} \le t$ and $\catBb{e^{E}_{t}} \le t$, and hence we obtain $(e_{m-1}^{E})^{\ast}(u) \!=\! 0$ and $\Max\{ t \!\ge\! 0 \midvert (e_{t-1}^{E})^{\ast}(u) \!=\! 0\} \ge m = \wgt[B]{u;R}$.
Conversely assume that $\Max\{ t \!\ge\! 0 \midvert (e_{t-1}^{E})^{\ast}(u) \!=\! 0\} =m$.
If $(\phi,f) : (F,q,Y) \to (E,p,X)$ satisfies $\catBb{\phi}<m$, then there exits an open cover of $F$ by at most $m$ open subsets $U_{i}$ on each of which $\phi|U_{i}$ is fibrewise homotopic to $s \comp f \!\comp q |_{U_{i}}$.
By standard arguments of homotopy theory, we may assume that there exits at most $m$ closed cofibrations $F_{i} \hookrightarrow F$. 
Then by extending the homotopy onto $F$ to obtain a fibrewise map $r : F \to \overset{m}{\fatvee_{\!B}}E \subset \overset{m+1}{\product_{B}}E$ a fibrewise compression of the fibrewise diagonal $\Delta^{m}_{B} : E \to \overset{m-1}{\product_{B}}E$, which gives a fibrewise map $\psi : F \to P_{\!B}^{m-1}(\LoopB{E})$ which is a lift of $\phi : F \to E$ on $e^{E}_{m-1} : P_{\!B}^{m-1}(\LoopB{E}) \to E$.
Since $(e^{E}_{m-1})^{*}(u)=0$, we have $\phi^{*}(u)=0$, and it implies (\ref{prop:multiplicativeformula-2}).
\par(\ref{prop:multiplicativeformula-3}) \ 
$\wgt[B]{E}=t \!\ge\! 0$ if and only if $(e_{t}^{E})^{\ast}$ is monic and $(e_{t-1}^{E})^{\ast}$ has non-trivial kernel, which is equivalent to that, for any $u \!\neq\! 0$ in $H_{B}^{\ast}(E;R)$, $(e_{t}^{E})^{\ast}(u) \!\neq\! 0$ but there is an element $v \!\in\! H_{B}^{\ast}(E;R)$ such that $(e_{t-1}^{E})^{\ast}(v) \!=\! 0$, in other words, $\wgt[B]{u} \!\leqq\! t$ for all $u \!\in\! H_{B}^{\ast}(E;R)$ but $\wgt[B]{v}=t$.
\end{proof}

\begin{prop}\label{catwgt}
We have
$\catBb{E} \ge \wgt[B]{E;R} \ge \cupl[B]{E;R}$.
\end{prop}
\begin{proof}
Let $\catBb{E} = t$. Then there is a map $s : E \to P^{t}_{\!B}(\LoopB{E})$ such that $e_{t}^{E}\comp s = \id_E$, which implies that $s^{\ast} \comp (e_{t}^{E})^{\ast} = (\id_E)^{\ast}$.
Thus, $(e_{t}^{E})^{\ast} : H_{B}^{\ast}(E;R) \to H_{B}^{\ast}(P^{t}_{\!B}(\LoopB{E});R)$ is monic, and hence $\catBb{E} \ge \wgt[B]{E;R}$.
The latter part is obtained from $\wgt[B]{u;R}\!\ge\! 1$ by Proposition \ref{prop:multiplicativeformula} (\ref{prop:multiplicativeformula-1}).
\end{proof}

\begin{remark}
Though it is not necessary in our arguments, \cite[Theorem 1.7 \& 1.10]{MR2556074} says the following equalities for the fibrewise well-pointed space $\double{X}=E_{\id_{X}}$, while we skip the details.
\par
\begin{enumerate*}{(1)}
\item
$\tc{X} = \catBb{\double{X}}$,
\hitem
$\cupl[\varpi]{X;R} = \cuplB{\double{X};R}$.
\end{enumerate*}\par
Also the proof of \cite[Theorem 1.11]{MR2556074} claiming $\wgt[\varpi]{u;R} = \wgtBB{u;R}$ for $u \!\in\! H_{B}^{\ast}(\double{X};R)$ works fine to obtain $\wgt[\varpi]{u;R} = \wgtB{u;R}$ for $u \!\in\! H_{B}^{\ast}(\double{X};R)$, while we skip the details, too.
\end{remark}

Using the above observations, we obtain the following.

\begin{thm}\label{thm:mainthm}
Let $m \ge 2$.
Assume that a group $G$ acts on $S^{m}$ freely, satisfying $\cat{X}=\wgt{X;\field_{2}}=m$ where $X=S^{m}/G$, and thus the weight of the generator $z \in H^{m}(X;\field_{2}) \cong \field_{2}$ is equal to $m$, where $\field_{2}$ denotes the prime field of characteristic $2$. 
Let $\widehat{z} = z \otimes z$ be the generator of $H_{B}^{2m}(\double{X};\field_{2}) = H^{2m}(X \times X,\Delta(X);\field_{2}) \cong \field_{2}$.
Then we have the following.
\begin{enumerate}
\item The three statements (\ref{enum:condition 1}), (\ref{enum:condition 2}) and (\ref{enum:condition 3}) below are equivalent.
\item The statements (\ref{enum:condition 1}), (\ref{enum:condition 2}) or (\ref{enum:condition 3}) implies $\catBb{\double{X}} = \tc{X} = 2m$.\par
\end{enumerate}
\begin{enumerate*}{(i)}
\vitem\label{enum:condition 1} $\Mwgt[B]{\double{X};\field_{2}} \!\ge\! 2m$,
\hitem\label{enum:condition 2} $\wgt[B]{\double{X};\field_{2}} \!\ge\! 2m$,
\hitem\label{enum:condition 3} $\wgt[B]{\widehat{z};\field_{2}} \!\ge\! 2m$.
\end{enumerate*}
\end{thm}
\begin{proof}
Since we know that $\wgt[B]{z \otimes z} \le \wgt[B]{X;\field_{2}} \le \Mwgt[B]{X;\field_{2}} \le \catBb{X} \le \cat{X \times X} \le \dim{X \times X}$ $=$ $2m$, it is straitforward to obtain (\ref{enum:condition 3}) $\Rightarrow$ (\ref{enum:condition 2}) $\Rightarrow$ (\ref{enum:condition 1}) $\Rightarrow$ $\catBb{X} = \tc{X}$ $=$ $2m$.
So we are left to show (\ref{enum:condition 1}) $\Rightarrow$ (\ref{enum:condition 3}):
let $\wgt[B]{z \otimes z} < 2m$ and $(e_{2m-1}^{E})^{\ast}(z \otimes z) \!\ne\! 0$. 
Let $W=X \times X$ and $W_{0}=W \smallsetminus D^{2m}$ the once punctured submanifold.
Then we have $W \homeo W_{0} \cup_{f} D^{2m}$, where $f$ is an attaching map, and hence we obtain $H^{*}(E;\field_{2}) \cong H^{*}(E_{0};\field_{2}) \oplus H^{*}(D^{2m},S^{2m-1};\field_{2})$, where the latter direct summand is isomorphic to $\field_{2}$ generated by $z \otimes z$.
Since $\catBb{E_{0}} \le 2m{-}1$, the inclusion $E_{0} \hookrightarrow E$ has a lift to $e^{E}_{2m-1} : P_{\!B}^{2m-1}(\LoopB{E}) \to E$.
Thus the direct summand $H^{*}(E_{0};\field_{2}) $ mapped to a direct summand of $H^{*}(P_{\!B}^{2m-1}(\LoopB{E});\field_{2})$ by $(e^{E}_{2m-1})^{*}$.
By the hypothesis, the entire module $H^{*}(E;\field_{2})$ mapped to a direct summand of $P_{\!B}^{2m-1}(\LoopB{E})$ by $(e^{E}_{2m-1})^{*}$, and hence we obtain $\Mwgt[B]{\double{X}} < 2m$.
\end{proof}

We remark that $\wgtB{z \otimes z}<2m$ may not imply $\tc{X}$ $<$ $2m$.
We must know about a higher Hopf invariant to show a result similar to that for L-S category.

From now on, $G$, $M$ and $p$ stands for $Q_8$, $S^3/Q_{8}$ and the canonical projection of the principal $Q_{8}$-bundle $S^{3} \twoheadrightarrow M$, respectively.
We show the following in \S \ref{sect:mainprop}.

\begin{thm}\label{thm:cohomology}
The generator $z \in H^{3}(M;\field_{2}) \cong \field_{2}$ satisfies $\wgt{z}=\cupl{z}=3$ which implies\par
\hfil$\cupl{M} = \wgt{M} = \cat{M} = \dim{M} = 3.$\hfil
\end{thm}

\begin{prop}\label{prop:mainprop}
	$5 \le \cuplB{\double{M}}\le \catBb{\double{M}} \le 2\cat{M} = 6$.
\end{prop}

The following statement is our main result.

\begin{thm}\label{thm:maincor}
$\tc{M} = \catBb{\double{M}} = 6$.
\end{thm}

In view of Theorem \ref{thm:mainthm} and Proposition \ref{prop:multiplicativeformula}, it is sufficient to show the following lemma. 

\begin{lem}\label{lem:mainlem}
	$(e_{5}^{\double{M}})^{\ast}(z \otimes z) \!=\! 0$ in $H^{\ast}(P^{5}_{\!B}(\LoopB{\double{M}});\field_{2})$. \end{lem}

\section{Proof of Theorem \ref{thm:cohomology} and Proposition \ref{prop:mainprop}}\label{sect:mainprop}

First, we introduce a modified Bar construction of a group $G$ as the realization of a nerve of the category of one object with morphism set $G$, where we only use face relations for the realization so as to obtain $C^{*}(G;\field_{2})=C^{*}(P^{\infty}G;\field_{2})$. Then by Segal \cite[Appendix A]{MR353298}, it follows that $P^{\infty}G \simeq K(G,1)$.
The cell-structure of the modified Bar construction is as follows:
$$
P^{\infty}G = \bigcup_{t \ge 0} P^{t}G,\quad P^{t}G= \underset{(g_{1},\ldots,g_{t}) \in \reduced{G}^{t}}{\textstyle\bigcup}\,\{g_{1}|g_{2}|\cdots|g_{t}\},
$$
where $\{g_{1}|\cdots|g_{t}\}$ is a $t$-simplex, $(g_{1},\ldots,g_{t}) \in \reduced{G}^{t}$, \ifx\undefined\unreduced $\reduced{G}=G \smallsetminus \{e\}$ \fi if $t \!\ge\! 1$, the unique $0$-simplex denoted by $\{\}$ if $t \!=\! 0$.
The boundary of $\{g_{1}|\cdots|g_{t}\}$, $t \!\ge\! 1$ is given by the following formulas:
\ifx\undefined\unreduced
\begin{align*}
\partial_{i}\{g_{1}|\cdots|g_{t}\} = \begin{cases}
\{g_{2}|\cdots|g_{t}\},&i=0,
\\
\{g_{1}|\cdots|g_{i-1}|g_{i}g_{i+1}|g_{i+2}|\cdots|g_{t}\},&0<i<t, \ g_{i}g_{i+1}\not=e,
\\
\{g_{1}|\cdots|g_{i-1}|g_{i+2}|\cdots|g_{t}\},&0<i<t, \ g_{i}g_{i+1}=e,
\\
\{g_{1}|\cdots|g_{t-1}\},&i=t.
\end{cases}
\end{align*}
\else
\begin{align*}
\partial_{i}\{g_{1}|\cdots|g_{t}\} = \begin{cases}
\{g_{2}|\cdots|g_{t}\},&i=0,
\\
\{g_{1}|\cdots|g_{i-1}|g_{i}g_{i+1}|g_{i+2}|\cdots|g_{t}\},&0<i<t, 
\\
\{g_{1}|\cdots|g_{t-1}\},&i=t.
\end{cases}
\end{align*}
\fi

The following is well-known (see \cite[(IV.2.10)]{MR2035696}) for $H^{*}(G;\field_{2})=H^{*}(P^{\infty}G;\field_{2})$.

\begin{fact}\label{thm:bar-resolution}
$H^{\ast}(P^{\infty}G;\field_2) = A^{*} \otimes \field_{2}[w]$, $A^{*}=\field_2[x,y]/(x^3,y^3,x^2{+}y^2{+}x{\cdot}y)$, $w \in H^{4}(P^{\infty}G;\field_{2})$ and $x, \,y \in H^{1}(P^{\infty}G;\field_{2})$, where $x\!=\![\alpha]$ and $y\!=\![\beta]$ are given by $\alpha\{a^{m}b^{n}\}=m$ and $\beta\{a^{m}b^{n}\}=n$ on $C_{1}(P^{\infty}G)$, respectively.
In particular, $H^{3}(P^{\infty}G;\field_2) \cong \field_2$ with a generator $z = x^{2}{\cdot}y = x{\cdot}y^{2}$.
\end{fact}

Second, we construct a cell complex $BG \simeq K(G,1)$ as a minimal complex realising the group cohomology $H^{*}(G;\field_{2})$ defined by K.~Fujii \cite{MR334184}, which is actually much smaller than $P^{\infty}G$:
taking quotient from $S^{4t+3}= \{(x_0,\dots ,x_{t}) \in \quaternion^{t+1} \midvert \vert{x_0}\vert^{2} {+} \cdots {+} \vert{x_{t}}\vert^{2} \!=\! 1\}$ by the following action of $G \subset \mathrm{Sp}(1)$, Fujii defined a series of manifolds $N^{t}(2)=S^{4t+3}/G$, $t \!\ge\! 0$:
\[g(x_0,\dots ,x_{t}) = (gx_0,\dots ,gx_{t})\]
for $g \in G \subset \mathrm{Sp}(1)$ and $(x_0,\dots ,x_{t}) \in S^{4t+3} \subset \quaternion^{t+1}$.

Since $\bigcup_{t=0}^{\infty} S^{4t+3}$ is contractible, the canonical projection $p_{\infty} : S^{\infty} = \bigcup_{t=0}^{\infty} S^{4t+3} \to \bigcup_{t=0}^{\infty}N^{t}(2) =:BG$ is a universal principal $G$-bundle.
This tells us that a canonical inclusion $i : M=N^{0}(2) \hookrightarrow BG$ gives a classifying map of the principal $G$-bundle $q : S^{3} \twoheadrightarrow M$, since $q=q_{\infty}|_{S^{3}}$.

\smallskip

We know the following finite cell decomposition of $N^t(2)$, $t \ge 0$.

\begin{fact}[Fujii {\cite[Page 253--254]{MR334184}}]\label{boundaryformula}
$S^{4t+3}$ is a $G$-cell complex with the cell decomposition
$$
\left\{ge^{4k}, \,ge^{4k+1}_{1}, \,ge^{4k+1}_{2}, \,ge^{4k+2}_{1}, \,ge^{4k+2}_{2}, \,ge^{4k+3} \mid 0 \!\le\! k \!\le\! t, \ g \in G \right\}
$$ 
whose boundary formulas in the cellular chain complex with coefficients in $\field_{2}$ is given as follows:
\begin{align*}
\partial &e^{0} = 0, \quad\partial
e^{4k+4} = \sum_{g\in G}ge^{4k+3}\, \ (k \ge 1), \\
\partial &e^{4k+1}_{1} = (a{-}1)e^{4k}, \quad\partial e^{4k+1}_{2} = (b{-}1)e^{4k}, \\
\partial &e^{4k+2}_{1} = (a {+} 1)e^{4k+1}_{1} - (b{+}1)e^{4k+1}_{2}, \quad%
\partial e^{4k+2}_{2} = (ab{+}1)e^{4k+1}_{1} + (a{-}1)e^{4k+1}_{2}, \\
\partial &e^{4k+3} = (a{-}1)e^{4k+2}_{1} - (ab{-}1)e^{4k+2}_{2}.
\end{align*}
\end{fact}

The above cell decomposition of $S^{4t+3}$ induces a cell decompostion of $N^{t}(2)$, $t \!\ge\! 0$ as follows.

\begin{fact}[Fujii {\cite[Lemma 2.1]{MR334184}}]\label{dec}
The manifold $N^{t}(2)$ can be decomposed as the finite cell complex whose cells are given by 
$$
\left\{e^{4k}, \,e^{4k+1}_1, \,e^{4k+1}_2, \,e^{4k+2}_1, \,e^{4k+2}_2, \,e^{4k+3} \mid 0\!\le\!k \!\le\!t\right\}
$$
	with the following boundary formulas in the cellular chain complex with coefficients in $\field_{2}$ associated with the cell decomposition of $N^{t}(2)$, $t \!\ge\! 0$, above:
\begin{align*}
&\partial e^{0} = 0,\quad \partial e^{4k} = 2^{3}e^{4k-1}_1\, \ (k \ge 1),\\
&\partial e^{4k+1}_1 = 0,\quad \partial e^{4k+1}_2 = 0,\\
&\partial e^{4k+2}_1 = 2e^{4k+1}_1 - 2e^{4k+1}_2, \quad \partial e^{4k+2}_2 = 2e^{4k+1}_1,\\
&\partial e^{4k+3} = 0.
\end{align*}
\end{fact}
We remark that the naturality of the above decomposition implies that $N^{k}(2)$ is a sub-complex of $N^{t}(2)$ if $0 \le k \le t$, and hence $M$ is a sub-complex of $BG$.

Then, we describe the cohomology groups of $N^{t}(2)$ with coefficients in $\field_{2}$.

\begin{fact}\label{prop:homologyandcohomology}
The cohomology groups of $BG = \bigcup_{t=0}^{\infty}N^{t}(2)$ are given as follows:
\begin{align*}
H^k(BG;\field_{2}) &=
\begin{cases}
\field_2 \oplus \field_2 & k\equiv 1,2 \mod 4,\\
\field_2 & k\equiv 3,0 \mod 4.
\end{cases}
\end{align*}
\end{fact}

Since $BG \simeq P^{\infty}G$, we obtain its multiplicative structure as $H^{\ast}(BG;\integral/(8)) \cong A^{*} \otimes \field_{2}[w]$.

\begin{prop}\label{prop:cohomologyofM}
$H^{\ast}(M;\field_2) \cong A^{*}$.
Further, we have $H^{3}(M;\field_2) \cong \field_2$ with $z = x^{2}{\cdot}y = x{\cdot}y^{2}$.
\end{prop}
\begin{proof}
The additive structures are obvious, and so we show the multiplicative structure for $\field_{2}$-coefficient.
Since $\pi_{1}(M) \cong G$, there is a classifying map $i : M \hookrightarrow BG$ inducing the following fibration with fibre $S^{3}$, where the map $S^{3} \twoheadrightarrow M$ is the universal covering of $M$:
\[S^3 \twoheadrightarrow M \overset{i}\hookrightarrow BG.\]
Since the action of $G$ preserves the orientation, the fibration is simple, and hence the $E_{2}$-term of the Serre spectral sequence with coefficients in $\field_{2}$ for the above fibration is described as follows:
$$
E_{2}^{p,q} \cong H^{p}(BG;\field_{2}) \otimes H^{q}(S^{3};\field_{2}) \cong A^{*} \otimes \field_{2}[w] \otimes \Lambda(s_{3}),
$$
where $s_{3}$ is the generator of $H^{3}(S^{3};\field_{2}) \cong \field_{2}$.
Then $E_{r}^{*,*}$, $r \!\ge\! 2$ has no non-trivial differential other than $d_4$ in the following diagram, where $H^{p} = H^{p}(BG;\field_{2})$.
\def\gap{20}
\newcount\xLine \xLine=\gap
\def\xcoodinate#1{
\put(\number\xLine,0){$#1$}\advance\xLine by \gap
}
\newcount\yLine \yLine=\gap
\def\ycoodinate#1{
\put(0,\number\yLine){$#1$}\advance\yLine by \gap
}
\def\xycoodinate(#1,#2)#3{
\xLine=\gap \multiply\xLine by #1
\yLine=\gap \multiply\yLine by #2
\advance\xLine by \gap\advance\yLine by \gap
\put(\number\xLine,\number\yLine){\makebox(6,6)[cc]{\small $#3$}}
}
\newcount\haba \haba=\gap
\multiply\haba by 11 \divide\haba by 2
\newcount\takasa \takasa=\gap \multiply\takasa by 9 \divide\takasa by 2
\par\vspace{1.5ex}
\begin{center}
\begin{picture}(\number\haba,\number\takasa)(0,0)
\put(0,10)	{\line(1,0){\number\haba}}
\put(10,0)	{\line(0,1){\number\takasa}}
\xcoodinate{0}\xcoodinate{1}\xcoodinate{2}\xcoodinate{3}\xcoodinate{4}
\ycoodinate{0}\ycoodinate{}\ycoodinate{}\ycoodinate{3}
\xycoodinate(0,0){\field_{2}1}
\xycoodinate(1,0){H^{1}}
\xycoodinate(2,0){H^{2}}
\xycoodinate(3,0){H^{3}}
\xycoodinate(4,0){\field_{2}w}
\xycoodinate(0,3){\field_{2}s_{3}}
\put(31,76){\vector(4,-3){60}}
\put(62,56){\small$d_{4}$}
\put(55,46){\small$\cong$}
\end{picture}
\end{center}
Since $w \in E_{2}^{4,1}$ can not survive in $E_{\infty}$-term, we have $w \in \Img d_{4}$, and hence $d_4(s_3) = w$.
Thus $H^{\ast}(M;\field_2)$ is isomorphic to $E_{5}^{*,*}=A^{*}$ and we have done.
\end{proof}

\begin{notation}
We denote by $\zeta$ and $1$ the cochains dual to $e^{3} \in C_{3}(M)$ and $e^{0} \in C_{0}(M)$, respectively, in $\hom(C^{*}(M);\field_{2})$.
Since $S^{3} = \bigcup\{\,g\sigma \,\vert\, g \!\in\! G\}$ where $\sigma$ runs over all cells of $M$, $p^{*}\zeta(g\sigma)=1$ if and only if $\sigma=e^{3}$ for any $g \in G$ and $p^{*}1(g\sigma)=1$ if and only if $\sigma=\ast$ for any $g \in G$.
\end{notation}

\begin{Claim}[Theorem \ref{thm:cohomology}]
Since $z=x^{2}{\cdot}y\not=0$ in $H^{*}(M;\field_{2})$, we have $\wgt{z}=\cupl{z}=3$, and hence we obtain the following.\par
\hfil$\cupl{M} = \wgt{M} = \cat{M} = \dim{M} = 3.$\hfil 
\end{Claim}

Let us denote $X =  x \otimes 1 + 1 \otimes x$ and $Y = y \otimes 1 + 1 \otimes y$, respectively, which are in the zero-divisors-ideal $\ker \Delta^{*} = H_{B}^{*}(\double{M};\field_{2})$.
Thus, we have $X^3{\cdot}Y^2 \neq 0$ and hence we obtain the following proposition.

\begin{Claim}[Proposition \ref{prop:mainprop}]
	$5 \le \cuplB{\double{M}}\le \catBb{\double{M}} = \tc{M} \le 2\cat{M} = 6$.
\end{Claim}

\section{Proof of Lemma \ref{lem:mainlem}}\label{sect:mainlem}

In general, $P^{t}_{\!B}\LoopB{\double{M}}$, \,$t \!\ge\! 0$,  the fibrewise projective $t$-space of the fiberwise loop space of $\double{M}$ has a misterious structure, while the following lemma is known (cf. \cite[Proposition 2.1]{MR3975098}).

\begin{lem}\label{lem:core}
There is a fibrewise homotopy equivalence $\tilde{f}_{0} : \LoopB{\double{BG}} \to S^{\infty}\times_{\ad{}}G$, such that
\begin{enumerate}
\item
$\tilde{f}_{0} : \LoopB{\double{BG}} \to S^{\infty}\times_{\ad{}}G$ is a fibrewise $A_{\infty}$-map.
\item
$\tilde{f}_{0}$ induces a fibrewise homotopy equivalence $\tilde{f}$ $:$ $P^{t}_{\!B}\LoopB{\double{BG}} \to S^{\infty}\times_{\ad{}}P^{t}G$, where $P^{t}_{\!B}\LoopB{\double{BG}}$ is the fibrewise projective $t$-space and
$P^{t}G$ is the $t$-skeleton of $P^{\infty}G$ with adjoint action of $G$ on $P^{\infty}G$ given by
\[
	h\{g_1|g_2|\cdots|g_t\} = \{hg_1\bar{h}|hg_2\bar{h}|\cdots | hg_t\bar{h}\},
\]
where $h \!\in\! G$ and $\{g_1|g_2|\cdots|g_t\}$ is a $t$-cell in $P^{\infty}G$ indexed by $(g_{1},\ldots,g_{t}) \in G^{t}$, $t \!\ge\! 0$.
\end{enumerate}
\end{lem}

\begin{proof}
We know, by a result of Benson \cite[Proposition 2.12.1]{MR1156302}, there is a fibrewise pointed homotopy equivalence $\tilde{f}_{0} : \LoopB{\double{BG}} \to S^{\infty}\times_{\ad{}}G$ over $BG$.
\begin{enumerate*}{(1)}
\vitem
Since $\tilde{f}_{\!0}$ satisfies that $\tilde{f}_{\!0}(\alpha \cdot \beta) = \tilde{f}_{\!0}(\alpha){\cdot} \tilde{f}_{\!0}(\beta)$ on each fibre and each fibre of $EG\times_{\ad{}}G$ is a discrete set, we obtain that $\tilde{f}_{\!0}$ is a fibrewise $A_{\infty}$-map.
\vitem
Thus $\tilde{f}_{\!0}$ induces a fibrewise maps $\tilde{f} : P^{t}_{\!B}\LoopB{\double{BG}} \to S^{\infty}\times_{\ad{}}P^{t}G$ and $\tilde{f}_{\!E} : E^{t}_{\!B}\LoopB{\double{BG}} \to S^{\infty}\times_{\ad{}}E^{t}G$ 
for all $t \ge 0$ and the following commutative diagram of fibrewise spaces.
$$
\begin{diagram}
\node{E^{t+1}_{\!B}\LoopB{\double{BG}}}
\arrow{s}
\arrow{e}
\node{P^{t}_{\!B}\LoopB{\double{BG}}}
\arrow{s}
\arrow{e}
\node{\double{BG}}
\arrow{s,=}
\\
\node{S^{\infty}\times_{\ad{}}E^{t+1}G}
\arrow{e}
\node{S^{\infty}\times_{\ad{}}P^{t}G}
\arrow{e}
\node{\double{BG}}
\end{diagram}
$$
By employing a similar arguments as in \cite[Proposition 2.1]{MR3975098}.
\end{enumerate*}
\end{proof}

By restricting the fibrewise structure to a subspace $M \overset{i}\hookrightarrow BG$, we obtain a fibrewise pointed space $E_{i}=\double{BG}|_{M}=\ifdefined\reviseone{(M \times BG,\proj_{1},M,(\id {\times} i)\comp\Delta)}\else{(BG \times M,\proj_{2},M,(i {\times} \id)\comp\Delta)}\fi$ and the following.

\begin{thm}\label{core}
There is a fibrewise $A_{\infty}$-map $f_{\!0} : \LoopB{E_{i}} \to S^{3}\times_{\ad{}}G$ over $M=S^{3}/G$, i.e,
\begin{enumerate}
\item[(i)] $f_{\!0} : \LoopB{E_{i}} \to S^{3}\times_{\ad{}}G$ is a fibrewise homotopy equivalence.
\item[(ii)] $f_{\!0}$ induces a fibrewise homotopy equivalence $f : P^{t}_{\!B}\LoopB{E_{i}} \to S^{3}\times_{\ad{}}P^{t}G$, $t \!\ge\! 0$.
\end{enumerate}
\end{thm}
\begin{proof}
\begin{figure}[!ht]
\[\xymatrix{
M	\ar@{^{(}->}[d]_{i} 
& S^3\times_{\ad{}}P^{t}G
	\ar@{^{(}->}[d]_{\iota}
	\ar[l]
& P^{t}_{\!B}\LoopB{E_{i}}
	\ar@/_15pt/@{->}[ll]
	\ar@{^{(}->}[d]
			\ar@{->}[l]^{f} 
\\
BG & S^{\infty}\times_{\ad{}}P^{t}G
	\ar[l] 
& P^{t}_{\!B}\LoopB{\double{BG}} 
 	\ar@/^15pt/@{->}[ll]
	   	\ar[l]_{\tilde{f}}
}\]\caption{}
\end{figure}
We know that $P^{t}_{\!B}\LoopB{E_{i}} = i^{*}(P^{t}_{\!B}\LoopB{\double{BG}})$ is a fibrewise space over $M$ given by the pull-back of $M \overset{i}\hookrightarrow BG \leftarrow P^{t}_{\!B}\LoopB{\double{BG}}$, and there is a fibrewise homotopy equivalence $f : P^{t}_{\!B}\LoopB{E_{i}} \to S^3\times_{\ad{}}P^{t}G$ since $S^3\times_{\ad{}}P^{t}G$ is the pull-back of $M \overset{i}\hookrightarrow BG \leftarrow S^{\infty}\times_{\ad{}}P^{t}G $ and $\tilde{f} : P^{t}_{\!B}\LoopB{\double{BG}} \to S^{\infty}\times_{\ad{}}P^{t}G$ is a fibrewise homotopy equivalence by Lemma \ref{lem:core}.
\end{proof}

Next, we give the boundary formulas for $S^{3}\times_{\ad{}}P^{t}G$, $t \!\ge\! 0$.
The cell structure of $S^{3}\times_{\ad{}}P^{t}G$ can be described by product cells of $S^{3}$ and $P^{t}G$ as follows.
\[ S^{3}\times_{\ad{}}P^{t}G  =\bigcup_{\sigma \in \{ \text{$G$-cells of $S^{3}$}\}, \,\omega \in \reduced{G}^{t}} \sigma{\times}\{\omega\}/\sim, \]
where `$\sim$' is given by $g\sigma{\times}\{g\omega\reduced{G}\} \sim \sigma{\times}\{\omega\}, \, g \!\in\! G$.
Let us denote by $[\sigma| \{\omega\}]$ the equivalence class of $\sigma{\times}\{\omega\}$.
By Proposition \ref{boundaryformula}, we obtain modulo $2$ boundary formulas for $S^{3}\times_{\ad{}}P^{t}G$.
\begin{prop}
The boundaries of product cells
\[ 
[e^{4k} | \{\omega\} ],  [e^{4k+1}_{1} | \{\omega\} ],  \,[e^{4k+1}_{2} | \{\omega\} ],  \,[e^{4k+2}_{1} | \{\omega\} ],  \,[e^{4k+2}_{2} | \{\omega\} ], \,[e^{4k+3} | \{\omega\} ],\quad k \ge 0,
\] 
are described as a chain in the chain complex with coefficients in $\field_{2}$ as follows.
\begin{align*}&
\partial[e^{0}| \{\omega\} ]=[e^{0}| \{\partial\omega\} ],
\\&
\partial[e^{4k} | \{\omega\} ] = \sum_{g\in G}[e^{4k-1} | \{\overline{g}\omega g\}]+[e^{4k} | \{\partial\omega\} ]\quad (k \ge 1),
\\&
\partial[e^{4k+1}_{1}| \{\omega\} ]=[e^{4k}| \{\omega\} ]+[e^{4k}| \{\bar{a}\omega a\} ]+[e^{4k+1}_{1}| \{\partial\omega\} ], \\
&\partial[e^{4k+1}_{2}| \{\omega\} ]=[e^{4k}| \{\omega\} ]+[e^{4k}| \{\bar{b}\omega b\} ]+[e^{4k+1}_{2}| \{\partial\omega\} ], 
\\&
\partial[e^{4k+2}_{1}|\{\omega\}]=[e^{4k+1}_{1}|\{\omega\}]+[e^{4k+1}_{2}|\{\omega\}]\\&\qquad\qquad\qquad\qquad+[e^{4k+1}_{1}|\{\bar{a}\omega a\}]+[e^{4k+1}_{2}|\{\bar{b}\omega b\}]+[e^{4k+2}_{1}| \{\partial\omega\}], 
\\&
\partial[e^{4k+2}_{2}|\{\omega\}]=[e^{4k+1}_{1}|\{\omega\}]+[e^{4k+1}_{2}|\{\omega\}]\\&\qquad\qquad\qquad\qquad+[e^{4k+1}_{1}|\{\overline{ab} \omega ab\} ]+[e^{4k+1}_{2}| \{\bar{a}\omega a\} ]+[e^{4k+2}_{2}| \{\partial\omega\} ],
\\&
\partial[e^{4k+3}|\{\omega\}]=[e^{4k+2}_{1}|\{\omega\}]+[e^{4k+2}_{2}|\{\omega\}]\\&\qquad\qquad\qquad\qquad+[e^{4k+2}_{2}|\{\overline{ab}\omega ab\}]+[e^{4k+2}_{1}|\{\bar{a}\omega a\}]+[e^{4k+3}|\{\partial\omega\}],
\end{align*}
\ifx\undefined\unreduced 
where we abbreviate $\displaystyle[\sigma | \{\partial\omega\}] = \overset{t}{\underset{\substack{0 \le i \le t\\\dim{\partial_{i} \omega} = t{-}1}}{\textstyle\sum}} (-1)^{i}[\sigma | \{ \partial_{i} \omega \}]$ for $\sigma$ a $G$-cell of $S^{3}$ and $\omega \in \reduced{G}^{t}$.
\else
where we abbreviate $\displaystyle[\sigma | \{\partial\omega\}] = \overset{t}{\underset{0 \le i \le t}{\textstyle\sum}} (-1)^{i}[\sigma | \{ \partial_{i} \omega \}]$ for $\sigma$ a $G$-cell of $S^{3}$ and $\omega \in \reduced{G}^{t}$.
\fi
\end{prop}
From Theorem \ref{core}, there is a fibrewise homotopy equivariance $f : P^{5}_{\!B}\LoopB{E_{i}} \xrightarrow{\simeq} S^{3}\times_{\ad{}}P^{5}G$.
On the other hand, we have a fibrewise map $\lambda=P^{5}_{\!B}\LoopB{\id_{M} \times i} : P^{5}_{\!B}\LoopB{\double{M}} \to P^{5}_{\!B}\LoopB{E_{i}}$ over $M$ in Diag. 2, which is commutative and $P^{5}_{\!B}\LoopB{E_{i}}$ is the pull-back of $M \overset{i}\hookrightarrow BG \leftarrow P^{5}_{\!B}\LoopB{\double{BG}}$.
\begin{figure}[h!]
\[\xymatrix{
M \ar@{=}[d]& P^{5}_{\!B}\LoopB{\double{M}}\ar[d]^{\lambda}\ar[l] \\
M & P^{5}_{\!B}\LoopB{E_{i}} \ar[l]
}\]\caption{}
\end{figure}

Now, let us denote by $C_{\ast}(X)$ the celluar chain complex with coefficients in $\field_{2}$ and by $C^{\ast}(X,R)=\hom(C_{\ast}(X),R)$ the cellular cochain complex with coefficients in $\field_{2}$-module $R$.
We define three cochains $c \in C^6(S^{3}\times_{\ad{}}P^{6}G;\field_{2}) = C^6(S^{3}\times_{\ad{}}P^{\infty}G;\field_{2})$, $c' \in C^5(S^{3}\times_{\ad{}}P^{5}G;\field_{2}) = C^5(S^{3}\times_{\ad{}}P^{\infty}G;\field_{2})$ and 
$v \in C^1(S^{3}\times_{\ad{}}P^{1}G;\field_{2}) = C^1(S^{3}\times_{\ad{}}P^{\infty}G;\field_{2})$ by the following equations on a generator $[\sigma| \{ h_1 | \cdots | h_{t}\}]$, $t \ge 0$ in $C_{*}(S^{3}\times_{\ad{}}P^{\infty}G)$:
\begin{align*}&
c[\sigma| \{ h_1 | \cdots | h_{t}\}] = p^{*}\zeta(\sigma){\cdot}\alpha^{2}\beta\{h_1 | \cdots | h_{t}\},
\\&
c'[\sigma| \{ h_1 | \cdots | h_{t}\}] = p^{*}\zeta(\sigma){\cdot}\alpha^{2}\{h_1 | \cdots | h_{t}\},
\\&
v[\sigma| \{ h_1 | \cdots | h_{t}\}] = p^{*}1(\sigma){\cdot}\beta\{h_1 | \cdots | h_{t}\}.
\end{align*}

Then we show that $c$, $c'$ and $v$ are well-defined:
since $a^mb^na^kb^l$ $=$ $a^{m+(-1)^nk}b^{n+l}$, for any $g=a^mb^n$ and $h=a^kb^l$, we obtain
\begin{center}
$gh\bar{g}$ 
$=$ $a^{m+(-1)^{n}k+(-1)^{l}m}b^{l}$, 
\end{center}
and hence $\alpha(gh\bar{g})=k=\alpha(h)$ and $\beta(gh\bar{g})=l=\beta(h)$ in $\field_{2}$.
Hence they satisfy
\begin{align*}&
c[g\sigma| {g}\{h_1 | \cdots | h_{t}\}] 
= p^{*}\zeta(g\sigma){\cdot}\alpha^2\beta\{gh_1\bar{g} | \cdots | gh_{t}\bar{g}\}
\\&\qquad\qquad\qquad\qquad= p^{*}\zeta(\sigma){\cdot}\alpha^2\beta\{h_1 | \cdots | h_{t}\} 
= c[\sigma| \{h_1 | \cdots | h_{t}\}],
\\&
c'\{g\sigma| {g}\{h_1 | \cdots | h_{t}\}\}
= p^{*}\zeta(g\sigma){\cdot}\alpha^{2}\{gh_1\bar{g} | \cdots | gh_{t}\bar{g}\}
\\&\qquad\qquad\qquad\qquad= p^{*}\zeta(\sigma){\cdot}\alpha^{2}\{h_1 | \cdots | h_{t}\} 
= c'[\sigma| \{h_1 | \cdots | h_{t}\}],
\\&
v[g\sigma| {g}\{h_1 | \cdots | h_{t}\}] 
= p^{*}1(g\sigma){\cdot}\beta\{gh_1\bar{g} | \cdots | gh_{t}\bar{g}\}
\\&\qquad\qquad\qquad\qquad= p^{*}1(\sigma){\cdot}\beta\{h_1 | \cdots | h_{t}\} 
= v[\sigma| \{h_1 | \cdots | h_{t}\}].
\end{align*}
Thus $c$, $c'$ and $v$ are well-defined. Since $z$, $\alpha$ and $\beta$ are cocycles, we have the following.

\begin{prop}
$c$, $c'$ and $v$ are cocycles satisfying $c=v \cupp c'$ in $C^{*}(S^{3}\times_{\ad{}}P^{\infty}G;\field_{2})$.
\end{prop}

\begin{prop}
$(e''_{\infty})^{\ast}(z \otimes z)= [c]$ in $H^{6}(S^{3}\times_{\ad{}}P^{\infty}G;\field_{2})$.
\end{prop}
\begin{proof}
Firstly, since the fibre bundle $\hat{p} : S^{3}\times_{\ad{}}P^{\infty}G \to M$ is fibrewise homotopy equivalent to the product bundle $\proj_{1} : M \times P^{\infty}G \to M$, their Serre spectral sequences are naturally isomorphic.
Secondly, $(e''_{\infty})^{\ast}(z \otimes z)$ is non-zero and in the image of $j^{*} : H^{6}(S^{3} \times_{\ad{}} P^{\infty}G,\hat{p}^{-1}(M^{(2)});\field_{2}) \to H^{6}(S^{3} \times_{\ad{}} P^{\infty}G;\field_{2})$, since it is trivial in $H^{6}(\hat{p}^{-1}(M^{(2)});\field_{2})$.
Thirdly, we have 
\begin{align*}&
H^{6}(S^{3} \times_{\ad{}} P^{\infty}G,\hat{p}^{-1}(M^{(2)});\field_{2}) = E^{3,3}_{1}(S^{3} \times_{\ad{}} P^{\infty}G) \\&\qquad\qquad\qquad\cong E^{3,3}_{1}(M \times P^{\infty}G) = \hom(C_{3}(M),H^{3}(P^{\infty}G;\field_{2})) \cong \field_{2}\phi, 
\end{align*}
where the generator $\phi$ is the homomorphism taking $\phi(e^{3})$ to the non-trivial element in $H^{3}(P^{\infty}G;\field_{2}) \cong \field_{2}z$, $z=[\alpha^{2}\beta]$.
Finally, $\phi$ is the corresponding homomorphism to $[c] \in H^{6}(S^{3} \times_{\ad{}} P^{\infty}G;\field_{2})$ (cf, \cite[XIII.4]{MR516508}) by definition.
Thus, we obtain $(e''_{\infty})^{\ast}(z \otimes z)=j^{*}\phi=[c]$.
\end{proof}

We consider the linear equation on $u \in C^{5}(S^3\times_{\ad{}}P^5G;\field_{2}) = C^{5}(S^3\times_{\ad{}}P^{\infty}G;\field_{2})$ as follows.
\begin{align}
\delta_{5} u = c\quad \text{in} \quad C^{6}(S^3\times_{\ad{}}P^5G;\field_{2}).\tag{Eq. \!1}\label{eq:Equation 1}
\end{align}

\begin{figure}[!ht]
\[\xymatrix{
  P^{5}_{\!B}\LoopB{\double{M}} \ar@/^15pt/@{->}[rr]^-{e^{\double{M}}_{5}} \ar@{^{(}->}[r]  \ar@{>}[d]_{\lambda} & P^{\infty}_{\!B}\LoopB{\double{M}} \ar[r] & \double{M} \ar[d]^{\id_{M} \times i}
  \\
 P^{5}_{\!B}\LoopB{E_{i}}\ar[d]_-{f} \ar@/^15pt/@{->}[rr]^-{e'_5} \ar@{^{(}->}[r] & \ar[r]_-{\simeq} P^{\infty}_{\!B}\LoopB{E_{i}} \ar[d]^{f} & E_{i} \ar@{=}[d]
 \\
S^{3}\times_{\ad{}}P^{5}G \ar@{^{(}->}[r] & S^{3}\times_{\ad{}}P^{\infty}G \ar[r]_-{e''_{\infty}} & E_{i}
}\]\caption{}
\end{figure}

Since (\ref{eq:Equation 1}) exceeds the acceptable size for our manipulation, we are forced to consider a similar but smaller linear equation on $u' \in C^4(S^3\times_{\ad{}}P^4G;\field_{2}) = C^4(S^3\times_{\ad{}}P^{\infty}G;\field_{2})$ as follows.
\begin{align}
\delta u' = c' \quad \text{in} \quad C^{5}(S^3\times_{\ad{}}P^4G;\field_{2}).
\tag{Eq. \!2}\label{eq:Equation 2}
\end{align}
Here, we may consider the equation (\ref{eq:Equation 2}) in $C^{5}(S^3\times_{\ad{}}P^5G;\field_{2})$ as $\delta u' = c'$ with indeterminacy on generators $[{*}|\{h_{1}|h_{2}|h_{3}|h_{4}|h_{5}\}]$ in $C_{5}(S^3\times_{\ad{}}P^5G) = C_{5}(S^3\times_{\ad{}}P^{\infty}G)$.
The following describes the relationship between (\ref{eq:Equation 1}) and (\ref{eq:Equation 2}).
\begin{prop}\label{prop:eq2toeq1}
If $u'$ is a solution for (\ref{eq:Equation 2}), then $u = v \cupp u'$ gives a solution for (\ref{eq:Equation 1}).
\end{prop}
\begin{proof}
Let $u' \in C^4(S^3\times_{\ad{}}P^4G;\field_{2}) = C^4(S^3\times_{\ad{}}P^{\infty}G;\field_{2})$ be a solution of (\ref{eq:Equation 2}).
Then, $v \cupp \delta u'$ coincides with $v \cupp c'$ in $C^6(S^3\times_{\ad{}}P^5G;\field_{2})$:
for any generator $[\sigma|\{h_{1}|\cdots|h_{k}\}]$ in $C_{6}(S^3\times_{\ad{}}P^5G)$, 
\begin{align*}&
(v \cupp \delta u')[\sigma|\{h_{1}|\cdots|h_{k}\}] 
= v[{\ast}|\{h_{1}\}] \cdot \delta u'[\sigma|\{h_{2}|\cdots|h_{k}\}]
\\&\qquad\qquad
= v[{\ast}|\{h_{1}\}] \cdot c'[\sigma|\{h_{2}|\cdots|h_{k}\}]
= (v \cupp c')[\sigma|\{h_{1}|\cdots|h_{k}\}]
\end{align*}
without indeterminacy, since $3 \!\le\! k \!\le\! 5$.
Then it implies $v \cupp \delta u' = v \cupp c'$ in $C^6(S^3\times_{\ad{}}P^5G;\field_{2})$, and we obtain $\delta u = \delta(v \cupp u') = v \cupp \delta u' = v \cupp c' = c$ in $C^6(S^3\times_{\ad{}}P^5G;\field_{2})$, since $\delta v = 0$.
\end{proof}

To show the existence of the solution $u'$ of (\ref{eq:Equation 2}), we use the matrix representations of homomorphisms $\kappa^{*}c'$, $u'$ and $\delta$ which are respectively denoted by a $1 \times d_{5}$ matrix $T_{\!c'}$, a $1 \times d_{4}$ matrix $T_{\!u'}$ and a $d_{5} \times d_{4}$ matrix $T_{\!\delta}$, where we denote $d_{i}=\dim[\field_{2}]{C^{i}(S^3\times_{\ad{}}P^4G;\field_{2})}$, 
\ifdefined\unreduced
namely, $d_{4} = 5256$ and $d_{5} = 9280$, while $\dim[\field_{2}]{C^{5}(S^3\times_{\ad{}}P^5G;\field_{2})} = 42048$.
\else
namely, $d_{4} = 3192$ and $d_{5} = 5537$, while $\dim[\field_{2}]{C^{5}(S^3\times_{\ad{}}P^5G;\field_{2})} = 22344$.
\fi
Then (\ref{eq:Equation 2}) is realized using matrices above as the following linear equation on $\mathbold{x}={}^tT_{\!u'}$:
\[\tag{Eq. \!M}\label{eq:Equation matrix}
{}^tT_{\!\delta}\,\mathbold{x} = {}^tT_{\!c'}.
\]
We remark that the existence of the solution for (\ref{eq:Equation matrix}) implies that for (\ref{eq:Equation 2}), and hence that for (\ref{eq:Equation 1}) by Proposition \ref{prop:eq2toeq1}.

Now let us recall the following statements from the theory of linear algebra.
\begin{thm}
There is a solution of (\ref{eq:Equation matrix}), provided that the following equation holds:
\[\tag{Eq. \!R}\label{eq:Equation rank}
\mbox{rank } {}^tT_{\!\delta} = \mbox{rank }\left[ {}^tT_{\!\delta} | {}^tT_{\!c'}\right], 
\]
where $\left[ {}^tT_{\!\delta} | {}^tT_{\!c'}\right]$ is the augmented coefficients matrix of the equation (\ref{eq:Equation matrix}).
\end{thm}

The size of the augmented coefficients matrix of (\ref{eq:Equation matrix}) exceeds the size for manipulation, we are forced to use a computer program to show (\ref{eq:Equation rank}).
Fortunately, our python program stops saying that (\ref{eq:Equation rank}) is correct.
Thus there is a solution $u'$ of (\ref{eq:Equation 2}), and we obtain $\delta u' = \kappa^{*}c'$.
Then by Proposition \ref{prop:eq2toeq1}, $u=v \cupp u'$ gives a solution of (\ref{eq:Equation 1}), and we obtain $[\kappa^{*}c] = [\delta u] = 0$.

\begin{Claim}[Lemma \ref{lem:mainlem}]
	$(e_{5}^{\double{M}})^{\ast}(z \otimes z) \!=\! 0$ in $H^{\ast}(P^{5}_{B}(\LoopB{\double{M}});\field_{2})$. \end{Claim}
\begin{proof}
Since ${i}^{*} : H^{*}({BG};\field_{2}) \to H^{*}({M};\field_{2})$ is surjective by the proof of Proposition \ref{prop:cohomologyofM}, we know that $z \otimes z = (i \times \id)^{*}(z \otimes z)$.
Then using a solution $u$ of (\ref{eq:Equation 1}), we obtain 
\begin{align*}
(e_{5}^{\double{M}})^{\ast}(z \otimes z) 
&= (e_{5}^{\double{M}})^{\ast}(\id_{M} \times i)^{\ast}(z \otimes z)
\\&= (f\!\comp\lambda)^{\ast}\kappa^{\ast}(e''_{\infty})^{\ast}(z \otimes z) 
= (f\!\comp\lambda)^{\ast}\kappa^{\ast}[c] 
\ifdefined\reviseone\else
\ \textem{ = (f\!\comp\lambda)^{\ast}[\delta (u) ] }
\fi
= 0,
\end{align*}
where $\kappa : S^3\times_{\ad{}}P^{5}G \hookrightarrow S^3\times_{\ad{}}P^{\infty}G$ denotes the canonical inclusion.

It implies $(e_{5}^{\double{M}})^{\ast}(z \otimes z) \!=\! 0$ in $H^{\ast}(P^{5}_{B}(\LoopB{\double{M}});\field_{2})$, and we are done. 
\end{proof}

\section*{Appendix A}

\ifdefined\algorithms
First, the following is the algorithm to check the equation (\ref{eq:Equation rank}).

\medskip

\begin{algorithm}[H]
\SetAlgoLined
\KwData{$G$ : Cayley table of $Q_8$, $L$ : the set of cells of $S^3/Q_8$, the modulo $2$ boundary formulas of product cells and $\delta u'$ : the cocycle in $C^{5}(S^3\times_{\ad{}}P^4G;\field_{2})$}
\KwResult{The status of the equation (\ref{eq:Equation rank})}
$Cells4$ : All $4$-cells in $S^3\times_{\ad{}}P^4G$\\
$Cells5$ : All $5$-cells in $S^3\times_{\ad{}}P^4G$\\
$delta$ : the matrix representing $\delta$ of type $(\#Cells5,\#Cells4)$
\\
$Delta$ : the augmeted matrix of type $(\#Cells5,\#Cells4 \!+\! 1)$ ($1$ for $c'$) 
\\[1ex]
We determine $Delta$ using the following algorithm.
\\[1ex]
$Delta = []$\\
\For{$s \in Cells5$ /* Comments below are in the case when $s = [e^{2}_{1}|\{a|a\} ]$ */}{
   $(f_s | c_s) = $ the tuple with column '$column$' /* $(-,\dots,-)$ */ \\
   \eIf{$(\delta u')(s) = 1$}{$(f_s | c_s)_{answer} = 1$ /* $(-,\dots,-,1)$ */}{$(f_s | c_s)_{answer} = 0$ /* $(-,\dots,-,0)$ */}
   \For{$t \in C_3$}{
\eIf{$t \in \partial s$/* $\partial [e^{2}_{1}|\{a|a\} ] = \{[e^{1}_{2}|\{a|a\}], [e^{1}_{2}|\{a^3|a^3\}], [e^{2}_{1}|\{a^2\} ] \}$ */}{
   $(f_s | c_s)_t = 1$
}{
   $(f_s | c_s)_t = 0$
}/*$(f_{[e^{2}_{1}|\{a|a\} ]} | c_{[e^{2}_{1}|\{a|a\} ]}) = (0,0,1,\dots, 0)$; in which $1$ is appearing at the place corresponding to $[e^{1}_{2}|\{a|a\}], [e^{1}_{2}|\{a^3|a^3\}]$ or $ [e^{2}_{1}|\{a^2\} ]$*/
   }
   /* the tuple $(f_s | c_s)$ consists of 0 or 1 */\\
   add $(f_s | c_s)$ to $Delta$}
\vskip1ex
This gives $Delta = [(f_{s_1} | c_{s_1}),\dots,(f_{s_n} | c_{s_n})]_{s_i \in Cells5}$.\\[1ex] We transform $Delta$ into the reduced row echelon form using Gaussian elimination.
\vspace*{1ex}
\caption{The algorithm to check the equation (\ref{eq:Equation rank})}
\end{algorithm}
\fi

\bigskip

\ifdefined\sourcecode

The following is the source code written in Python 3 to compute the rank of the differential $\delta : C^{4}(S^3\times_{\ad{}}P^4G;\field_{2}) \to C^{5}(S^3\times_{\ad{}}P^4G;\field_{2})$.

\medskip

\begin{lstlisting}[basicstyle=\ttfamily\footnotesize, frame=single, breaklines=true]
import numpy as np
from itertools import count, product

# z in H^{Element}(P^{Head}\Omega{M})
Manifold = 3    # = dim{M}
Element = 2     # = deg{z} =< dim{M}
Head = 4        # >= dim{M}

# Special Cocycles
def u(a) : # x
    return a % 2
def v(a): # y
    return a // 4
def z(i): # x^2
     return u(i[0]) * u(i[1])

# Cell decomposition of M
SpaceForm = [e0,e11,e12,e21,e22,e3] = ["e0","e11","e12","e21","e22","e3"]

# Group Structure
Group = [e,x,xx,xxx,y,xy,xxy,xxxy] = [0,1,2,3,4,5,6,7]
Mul = [
    [0,1,2,3,4,5,6,7],
    [1,2,3,0,5,6,7,4],
    [2,3,0,1,6,7,4,5],
    [3,0,1,2,7,4,5,6],
    [4,7,6,5,2,1,0,3],
    [5,4,7,6,3,2,1,0],
    [6,5,4,7,0,3,2,1],
    [7,6,5,4,1,0,3,2]]
Inv = [e,xxx,xx,x,xxy,xxxy,y,xy]

# Adjoint Action
def adj(a,b):
    return Mul[Inv[a]][Mul[b][a]]
def adl(a,lst):
    return [adj(a,i) for i in lst]
        
# Boundaries of the fibre - the resolution of the group
def difi(i,lst):
    n = len(lst)
    if i == 0:
        return lst[1:]
    elif i == n:
        return lst[:n-1]
    else:
        l = lst[:]
        l[i-1:i+1] = [Mul[lst[i-1]][lst[i]]]
        return l
def partial(s,lst):
    ans1 = [(s,difi(i,lst)) for i in range(len(lst)+1)]
    ans2 = []
    for i in ans1:
        if i in ans2:
            ans2.remove(i)
        else:
            ans2.append(i)
    return ans2

# Boundaries of the Fibrewise resolution
def boundary(s,w):
    if s == e3:
        return [(e21,w),(e21,adl(x,w)),(e22,w),(e22,adl(Mul[x][y],w))] + partial(s,w)
    elif s == e21:
        return [(e11,w),(e11,adl(x,w)),(e12,w),(e12,adl(y,w))] + partial(s,w)
    elif s == e22:
        return [(e11,w),(e11,adl(Mul[x][y],w)),(e12,w),(e12,adl(x,w))] + partial(s,w)
    elif s == e11:
        return [(e0,w),(e0,adl(x,w))] + partial(s,w)
    elif s == e12:
        return [(e0,w),(e0,adl(y,w))] + partial(s,w)
    elif s == e0:
        return partial(s,w)
    else:
        return 'miss!'

# Cells of the Reduced Projective Space up to dimension 'degree'
NonDegenerates = [7, 3, 6, 2, 5, 1, 4] # 823
lst = [[[]],[],[],[],[],[],[],[],[]]
for a in range(Head):
    for i,j in product(NonDegenerates,lst[a]):
        if j == []:
            w = [i]
        else:
            w = [i] + j
        lst[a+1].append(w)

# Counting Cells of each dimension.
Degree = Element + Manifold	# upper degree
degree = Degree - 1		# lower degree
#
LowerCells = [] # degree cells
for i in lst[degree]:
    LowerCells.append(str((e0,i)))
for i in lst[degree-1]:
    LowerCells.append(str((e11,i)))
    LowerCells.append(str((e12,i)))
for i in lst[degree-2]:
    LowerCells.append(str((e21,i)))
    LowerCells.append(str((e22,i)))
for i in lst[degree-3]:
    LowerCells.append(str((e3,i)))
#
UpperCells = [] # Degree cells
for i in lst[Degree]:
    UpperCells.append(str((e0,i)))
for i in lst[Degree-1]:
    UpperCells.append(str((e11,i)))
    UpperCells.append(str((e12,i)))
for i in lst[Degree-2]:
    UpperCells.append(str((e21,i)))
    UpperCells.append(str((e22,i)))
for i in lst[Degree-3]:
    UpperCells.append(str((e3,i)))

TheNumberofLowerCells = len(LowerCells)
TheNumberofUpperCells = len(UpperCells)

LabeltoNum = {LowerCells[i]:i for i in range(TheNumberofLowerCells)}

print('The Number of ' + str(degree) + '-cells is ' + str(TheNumberofLowerCells) + '.')
print('The Number of ' + str(Degree) + '-cells is ' + str(TheNumberofUpperCells) + '.')

# Initializing Augmented Coefficients Matrix
TheNumberofLowerCellsPlusAns = TheNumberofLowerCells + 1
TheNumberofUpperCellsPlusAns = TheNumberofUpperCells + 1

print("The size of the coefficients matrix delta is {0}x{1}.".format(TheNumberofUpperCells,TheNumberofLowerCells))

# Constructing Delta the set of coodinates whose entries are 1.
num = 0
Delta = [] # the set of coordinates whose entries are 1.
for i in lst[Degree]: # (e0,i) in UpperCells
    c = []
    J = boundary(e0,i)
    for k in boundary(e0,i):
        if k in c:
            c.remove(k)
        else:
            c.append(k)
    for j in c:
        j = str(j)
        if j in LowerCells:
            Delta.append((num,LabeltoNum[j]))
    num += 1
for i in lst[Degree-1]: # (e11,i) and (e12,i) in UpperCells
    c = []
    for k in boundary(e11,i):
        if k in c:
            c.remove(k)
        else:
            c.append(k)
    for j in c:
        j = str(j)
        if j in LowerCells:
            Delta.append((num,LabeltoNum[j]))
    num += 1
    c = []
    for k in boundary(e12,i):
        if k in c:
            c.remove(k)
        else:
            c.append(k)
    for j in c:
        j = str(j)
        if j in LowerCells:
            Delta.append((num,LabeltoNum[j]))
    num += 1
for i in lst[Degree-2]: # (e21,i) and (e22,i) in UpperCells
    c = []
    for k in boundary(e21,i):
        if k in c:
            c.remove(k)
        else:
            c.append(k)
    for j in c:
        j = str(j)
        if j in LowerCells:
            Delta.append((num,LabeltoNum[j]))
    num += 1
    c = []
    for k in boundary(e22,i):
        if k in c:
            c.remove(k)
        else:
            c.append(k)
    for j in c:
        j = str(j)
        if j in LowerCells:
            Delta.append((num,LabeltoNum[j]))
    num += 1
for i in lst[Degree-3]: # (e3,i) in UpperCells
    ans = z(i)
    if ans == 1:
        Delta.append((num,TheNumberofLowerCells))
    c = []
    for k in boundary(e3,i):
        if k in c:
            c.remove(k)
        else:
            c.append(k)
    for j in c:
        j = str(j)
        if j in LowerCells:
            Delta.append((num,LabeltoNum[j]))
    num += 1

AugCMat = np.zeros((TheNumberofUpperCells,TheNumberofLowerCellsPlusAns),dtype = bool)  # Augmented Coefficients Matrix
for i in Delta:
    AugCMat[i[0]][i[1]] = 1
CMat = np.delete(AugCMat,TheNumberofLowerCells,1)

print("The size of the augmented coefficients matrix Delta is {0}x{1}.".format(TheNumberofUpperCells,TheNumberofLowerCellsPlusAns))

# Transform AugCMat into the reduced row echelon form.
EAugCMat = AugCMat  # row echelon form of AugCMat
count = TheNumberofUpperCells
for j in range(0,TheNumberofLowerCells):
    num = -1
    for i in range(0,TheNumberofUpperCells):
        if num == -1: # no i st M[i][j] = 1
            if i < count:
                if EAugCMat[i][j] == 1:
                    num = i
            else:
                break
        else: # For i > num, add EAugCMat[num] to EAugCMat[i]
            if EAugCMat[i][j] == 1:
                EAugCMat[i] = EAugCMat[i] ^ EAugCMat[num]
    if num > -1: # move num culumn to the bottom
        EAugCMat = np.append(EAugCMat,[EAugCMat[num]],axis = 0)
        EAugCMat = np.delete(EAugCMat,num,0)
        count -= 1 # one more principal line
RankCMat = TheNumberofUpperCells - count
# row echelon form of CMat
ECMat = np.delete(EAugCMat,TheNumberofLowerCells,1)  
for i in range(0,count):
    ECMat = np.delete(ECMat, 0, 0)
print("The rank of the matrix delta is {0}.".format(RankCMat))
num = -1
for i in range(0,TheNumberofUpperCells):
    if num == -1: # no i st M[i][TheNumberofLowerCells] = 1
        if i < count:
            if EAugCMat[i][TheNumberofLowerCells] == 1:
                num = i
        else:
            break
    else: # For i > num, add EAugCMat[num] to EAugCMat[i]
        if EAugCMat[i][TheNumberofLowerCells] == 1:
            EAugCMat[i] = EAugCMat[i] ^ EAugCMat[num]
if num > -1: # move num culumn to the bottom
    EAugCMat = np.append(EAugCMat,[EAugCMat[num]],axis = 0) 
    EAugCMat = np.delete(EAugCMat,num,0) 
    count -= 1 # one more principal line
RankAugCMat = TheNumberofUpperCells - count
# row echelon form of AugCMat
for i in range(0,count):
    EAugCMat = np.delete(EAugCMat, 0, 0)
print("The rank of the matrix Delta is {0}.".format(RankAugCMat))

# Display a special solution
if RankCMat == RankAugCMat:
    Sol=[]
    for i in range(RankCMat):
        b = EAugCMat[i][TheNumberofLowerCells]
        if b == 1:
            for j in range(0,TheNumberofLowerCells):
                if ECMat[i][j] == 1:
                    Sol.append(j)
                    break
    LengthofSolution = len(Sol)
    if LengthofSolution > 0:
        print("The length of one particular solution is {0}.".format(LengthofSolution))
        print("The particular solution is {0}".format(LowerCells[Sol[0]]), end='')
        for i in range(1,len(Sol)):
            print(" + {0}".format(LowerCells[Sol[i]]), end='')
        print(".")
# Verify the solution
    for i in range(TheNumberofUpperCells):
        num = AugCMat[i][TheNumberofLowerCells]
        for j in Sol:
            num = num ^ CMat[i][j]
        if num != 0:
            print('NG!')
            break
    if i == TheNumberofUpperCells-1:
        print('OK!')
\end{lstlisting}

The above python program produces the following outputs:
\else

We use python program which produces the following outputs:
\fi

\par\smallskip
\ifdefined\unreduced 
\begin{lstlisting}[basicstyle=\ttfamily\footnotesize, frame=single, breaklines=true]
The Number of 4-cells is 5256.
The Number of 5-cells is 9280.
The size of the coefficients matrix delta is 9280x5256.
The size of the augmented coefficients matrix Delta is 9280x5257.
The rank of the matrix delta is 3600.
The rank of the matrix Delta is 3600.
\end{lstlisting}
\else
\begin{lstlisting}[basicstyle=\ttfamily\footnotesize, frame=single, breaklines=true]
The Number of 4-cells is 3192.
The Number of 5-cells is 5537.
The size of the coefficients matrix delta is 5537x3192.
The size of the augmented coefficients matrix Delta is 5537x3193.
The rank of the matrix delta is 2214.
The rank of the matrix Delta is 2214.
The length of one particular solution is 823.
\end{lstlisting}
\fi

\ifdefined\solution
The program outputs a list of cells in the form of {\small\verb$('e0', [7, 7, 7, 6])$}.
In fact, $[0,1,2,3,4,5,6,7]=[e,a,a^{2},a^{3},b,ab,a^{2}b,a^{3}b]$ in our program.
To compress the list, we transform them into the form of $[e0|7|7|7|6]$ which means the cell $[e^0|\{a^3b|a^3b|a^3b|a^2b\}]$.
Then one particular solution is given by $u'_{*} = [e0|{7|7|7|6}] + [e0|{7|7|6|6}] + [e0|{7|7|6|2}] + [e0|{7|7|2|5}] + [e0|{7|7|2|1}] + [e0|{7|7|5|3}] + [e0|{7|7|5|6}] + [e0|{7|7|5|5}] + [e0|{7|7|5|1}] + [e0|{7|7|1|3}] + [e0|{7|7|1|2}] + [e0|{7|3|6|5}] + [e0|{7|3|6|1}] + [e0|{7|3|5|3}] + [e0|{7|3|1|5}] + [e0|{7|6|7|5}] + [e0|{7|6|7|1}] + [e0|{7|6|3|3}] + [e0|{7|6|3|2}] + [e0|{7|6|3|5}] + [e0|{7|6|3|1}] + [e0|{7|6|6|2}] + [e0|{7|6|6|1}] + [e0|{7|6|2|7}] + [e0|{7|6|2|3}] + [e0|{7|6|2|6}] + [e0|{7|6|2|2}] + [e0|{7|6|2|1}] + [e0|{7|6|2|4}] + [e0|{7|6|5|3}] + [e0|{7|6|5|4}] + [e0|{7|6|1|7}] + [e0|{7|6|1|3}] + [e0|{7|6|1|2}] + [e0|{7|6|1|4}] + [e0|{7|2|7|6}] + [e0|{7|2|3|7}] + [e0|{7|2|3|2}] + [e0|{7|2|3|4}] + [e0|{7|2|6|6}] + [e0|{7|2|6|2}] + [e0|{7|2|2|6}] + [e0|{7|2|2|2}] + [e0|{7|2|2|5}] + [e0|{7|2|2|4}] + [e0|{7|2|5|2}] + [e0|{7|2|5|5}] + [e0|{7|2|5|1}] + [e0|{7|2|1|7}] + [e0|{7|2|1|3}] + [e0|{7|2|1|6}] + [e0|{7|2|1|4}] + [e0|{7|2|4|2}] + [e0|{7|5|7|6}] + [e0|{7|5|3|7}] + [e0|{7|5|3|2}] + [e0|{7|5|3|4}] + [e0|{7|5|6|3}] + [e0|{7|5|6|6}] + [e0|{7|5|6|2}] + [e0|{7|5|6|5}] + [e0|{7|5|2|2}] + [e0|{7|5|2|4}] + [e0|{7|5|5|3}] + [e0|{7|5|5|6}] + [e0|{7|5|5|2}] + [e0|{7|5|5|1}] + [e0|{7|5|1|3}] + [e0|{7|5|1|4}] + [e0|{7|5|4|2}] + [e0|{7|1|7|3}] + [e0|{7|1|7|2}] + [e0|{7|1|7|5}] + [e0|{7|1|3|3}] + [e0|{7|1|3|5}] + [e0|{7|1|6|3}] + [e0|{7|1|2|3}] + [e0|{7|1|2|6}] + [e0|{7|1|2|2}] + [e0|{7|1|2|4}] + [e0|{7|1|5|2}] + [e0|{7|1|5|5}] + [e0|{7|1|5|4}] + [e0|{7|1|1|3}] + [e0|{7|1|1|4}] + [e0|{7|1|4|3}] + [e0|{7|1|4|2}] + [e0|{7|1|4|5}] + [e0|{7|4|7|1}] + [e0|{7|4|3|7}] + [e0|{7|4|3|6}] + [e0|{7|4|6|7}] + [e0|{7|4|6|3}] + [e0|{7|4|6|6}] + [e0|{7|4|2|2}] + [e0|{7|4|2|4}] + [e0|{7|4|5|2}] + [e0|{7|4|5|4}] + [e0|{7|4|1|2}] + [e0|{7|4|4|2}] + [e0|{3|7|7|3}] + [e0|{3|7|7|2}] + [e0|{3|7|7|5}] + [e0|{3|7|7|1}] + [e0|{3|7|3|7}] + [e0|{3|7|3|3}] + [e0|{3|7|3|6}] + [e0|{3|7|3|5}] + [e0|{3|7|3|4}] + [e0|{3|7|6|5}] + [e0|{3|7|6|1}] + [e0|{3|7|2|3}] + [e0|{3|7|2|5}] + [e0|{3|7|5|7}] + [e0|{3|7|5|2}] + [e0|{3|7|5|5}] + [e0|{3|7|1|6}] + [e0|{3|7|1|4}] + [e0|{3|7|4|5}] + [e0|{3|7|4|1}] + [e0|{3|3|6|7}] + [e0|{3|3|2|1}] + [e0|{3|3|5|7}] + [e0|{3|3|5|3}] + [e0|{3|3|5|5}] + [e0|{3|3|1|6}] + [e0|{3|3|1|1}] + [e0|{3|6|7|3}] + [e0|{3|6|7|6}] + [e0|{3|6|7|2}] + [e0|{3|6|7|5}] + [e0|{3|6|3|7}] + [e0|{3|6|3|3}] + [e0|{3|6|3|6}] + [e0|{3|6|3|5}] + [e0|{3|6|3|4}] + [e0|{3|6|6|3}] + [e0|{3|6|6|6}] + [e0|{3|6|6|1}] + [e0|{3|6|6|4}] + [e0|{3|6|2|7}] + [e0|{3|6|2|3}] + [e0|{3|6|2|5}] + [e0|{3|6|2|1}] + [e0|{3|6|5|3}] + [e0|{3|6|5|5}] + [e0|{3|6|5|1}] + [e0|{3|6|1|7}] + [e0|{3|6|1|2}] + [e0|{3|6|1|5}] + [e0|{3|6|1|4}] + [e0|{3|6|4|7}] + [e0|{3|6|4|5}] + [e0|{3|2|7|1}] + [e0|{3|2|3|7}] + [e0|{3|2|3|6}] + [e0|{3|2|6|6}] + [e0|{3|2|6|2}] + [e0|{3|2|6|5}] + [e0|{3|2|6|4}] + [e0|{3|2|2|2}] + [e0|{3|2|2|5}] + [e0|{3|2|2|1}] + [e0|{3|2|2|4}] + [e0|{3|2|5|7}] + [e0|{3|2|5|3}] + [e0|{3|2|5|6}] + [e0|{3|2|5|2}] + [e0|{3|2|5|5}] + [e0|{3|2|5|1}] + [e0|{3|2|1|2}] + [e0|{3|2|1|1}] + [e0|{3|2|4|2}] + [e0|{3|5|7|7}] + [e0|{3|5|7|6}] + [e0|{3|5|7|1}] + [e0|{3|5|3|3}] + [e0|{3|5|3|6}] + [e0|{3|5|3|5}] + [e0|{3|5|3|4}] + [e0|{3|5|6|6}] + [e0|{3|5|6|5}] + [e0|{3|5|6|1}] + [e0|{3|5|2|7}] + [e0|{3|5|2|2}] + [e0|{3|5|2|5}] + [e0|{3|5|2|4}] + [e0|{3|5|5|5}] + [e0|{3|5|5|4}] + [e0|{3|5|1|5}] + [e0|{3|5|4|2}] + [e0|{3|5|4|5}] + [e0|{3|5|4|1}] + [e0|{3|1|7|1}] + [e0|{3|1|3|7}] + [e0|{3|1|3|6}] + [e0|{3|1|3|5}] + [e0|{3|1|6|3}] + [e0|{3|1|6|6}] + [e0|{3|1|6|5}] + [e0|{3|1|2|6}] + [e0|{3|1|2|2}] + [e0|{3|1|2|4}] + [e0|{3|1|5|7}] + [e0|{3|1|5|6}] + [e0|{3|1|5|4}] + [e0|{3|1|1|2}] + [e0|{3|1|4|2}] + [e0|{3|4|6|5}] + [e0|{3|4|6|1}] + [e0|{3|4|5|3}] + [e0|{3|4|1|5}] + [e0|{6|7|7|3}] + [e0|{6|7|3|3}] + [e0|{6|7|3|6}] + [e0|{6|7|3|2}] + [e0|{6|7|5|3}] + [e0|{6|7|5|2}] + [e0|{6|7|5|4}] + [e0|{6|7|1|5}] + [e0|{6|7|4|3}] + [e0|{6|7|4|5}] + [e0|{6|3|7|2}] + [e0|{6|3|7|5}] + [e0|{6|3|7|1}] + [e0|{6|3|3|2}] + [e0|{6|3|3|5}] + [e0|{6|3|3|1}] + [e0|{6|3|6|3}] + [e0|{6|3|6|5}] + [e0|{6|3|6|1}] + [e0|{6|3|6|4}] + [e0|{6|3|2|3}] + [e0|{6|3|2|2}] + [e0|{6|3|2|1}] + [e0|{6|3|5|7}] + [e0|{6|3|5|6}] + [e0|{6|3|5|2}] + [e0|{6|3|1|2}] + [e0|{6|3|1|5}] + [e0|{6|3|1|4}] + [e0|{6|3|4|7}] + [e0|{6|3|4|3}] + [e0|{6|3|4|2}] + [e0|{6|3|4|1}] + [e0|{6|3|4|4}] + [e0|{6|6|7|5}] + [e0|{6|6|3|7}] + [e0|{6|6|3|4}] + [e0|{6|6|2|5}] + [e0|{6|6|2|4}] + [e0|{6|6|5|3}] + [e0|{6|6|5|6}] + [e0|{6|6|5|1}] + [e0|{6|6|1|3}] + [e0|{6|6|4|1}] + [e0|{6|6|4|4}] + [e0|{6|2|7|7}] + [e0|{6|2|7|2}] + [e0|{6|2|6|7}] + [e0|{6|2|6|1}] + [e0|{6|2|2|7}] + [e0|{6|2|2|2}] + [e0|{6|2|2|4}] + [e0|{6|2|5|7}] + [e0|{6|2|5|6}] + [e0|{6|2|5|2}] + [e0|{6|2|1|7}] + [e0|{6|2|1|3}] + [e0|{6|2|1|6}] + [e0|{6|2|1|2}] + [e0|{6|2|1|5}] + [e0|{6|2|1|1}] + [e0|{6|2|4|7}] + [e0|{6|2|4|2}] + [e0|{6|2|4|1}] + [e0|{6|2|4|4}] + [e0|{6|5|7|3}] + [e0|{6|5|7|2}] + [e0|{6|5|7|5}] + [e0|{6|5|7|1}] + [e0|{6|5|3|3}] + [e0|{6|5|3|6}] + [e0|{6|5|3|2}] + [e0|{6|5|3|5}] + [e0|{6|5|3|1}] + [e0|{6|5|6|7}] + [e0|{6|5|6|6}] + [e0|{6|5|6|5}] + [e0|{6|5|6|4}] + [e0|{6|5|2|6}] + [e0|{6|5|2|2}] + [e0|{6|5|2|5}] + [e0|{6|5|5|7}] + [e0|{6|5|5|3}] + [e0|{6|5|5|6}] + [e0|{6|5|5|2}] + [e0|{6|5|5|5}] + [e0|{6|5|5|4}] + [e0|{6|5|1|2}] + [e0|{6|5|1|1}] + [e0|{6|5|1|4}] + [e0|{6|5|4|6}] + [e0|{6|5|4|2}] + [e0|{6|5|4|5}] + [e0|{6|1|7|3}] + [e0|{6|1|7|2}] + [e0|{6|1|3|6}] + [e0|{6|1|3|5}] + [e0|{6|1|6|3}] + [e0|{6|1|6|6}] + [e0|{6|1|6|2}] + [e0|{6|1|6|1}] + [e0|{6|1|2|7}] + [e0|{6|1|2|2}] + [e0|{6|1|2|5}] + [e0|{6|1|2|1}] + [e0|{6|1|5|7}] + [e0|{6|1|5|2}] + [e0|{6|1|5|5}] + [e0|{6|1|5|4}] + [e0|{6|1|4|6}] + [e0|{6|1|4|5}] + [e0|{6|4|7|7}] + [e0|{6|4|7|2}] + [e0|{6|4|7|5}] + [e0|{6|4|3|3}] + [e0|{6|4|3|2}] + [e0|{6|4|6|7}] + [e0|{6|4|6|3}] + [e0|{6|4|6|1}] + [e0|{6|4|2|7}] + [e0|{6|4|2|3}] + [e0|{6|4|2|2}] + [e0|{6|4|5|6}] + [e0|{6|4|5|1}] + [e0|{6|4|1|7}] + [e0|{6|4|1|6}] + [e0|{6|4|1|2}] + [e0|{6|4|1|5}] + [e0|{6|4|1|1}] + [e0|{6|4|4|7}] + [e0|{6|4|4|2}] + [e0|{6|4|4|5}] + [e0|{2|7|7|5}] + [e0|{2|7|3|7}] + [e0|{2|7|3|2}] + [e0|{2|7|3|4}] + [e0|{2|7|6|7}] + [e0|{2|7|6|6}] + [e0|{2|7|6|2}] + [e0|{2|7|6|4}] + [e0|{2|7|2|3}] + [e0|{2|7|2|6}] + [e0|{2|7|2|2}] + [e0|{2|7|2|5}] + [e0|{2|7|2|4}] + [e0|{2|7|5|3}] + [e0|{2|7|5|6}] + [e0|{2|7|5|2}] + [e0|{2|7|5|1}] + [e0|{2|7|1|5}] + [e0|{2|7|1|4}] + [e0|{2|7|4|2}] + [e0|{2|7|4|5}] + [e0|{2|7|4|1}] + [e0|{2|3|7|7}] + [e0|{2|3|7|3}] + [e0|{2|3|7|1}] + [e0|{2|3|3|7}] + [e0|{2|3|3|6}] + [e0|{2|3|3|1}] + [e0|{2|3|6|3}] + [e0|{2|3|6|6}] + [e0|{2|3|6|2}] + [e0|{2|3|6|4}] + [e0|{2|3|2|6}] + [e0|{2|3|2|2}] + [e0|{2|3|2|5}] + [e0|{2|3|2|1}] + [e0|{2|3|2|4}] + [e0|{2|3|5|3}] + [e0|{2|3|5|6}] + [e0|{2|3|5|2}] + [e0|{2|3|5|1}] + [e0|{2|3|1|6}] + [e0|{2|3|1|2}] + [e0|{2|3|4|2}] + [e0|{2|6|7|7}] + [e0|{2|6|7|3}] + [e0|{2|6|7|2}] + [e0|{2|6|7|5}] + [e0|{2|6|3|7}] + [e0|{2|6|3|5}] + [e0|{2|6|3|4}] + [e0|{2|6|6|7}] + [e0|{2|6|6|5}] + [e0|{2|6|6|1}] + [e0|{2|6|6|4}] + [e0|{2|6|2|7}] + [e0|{2|6|2|2}] + [e0|{2|6|2|5}] + [e0|{2|6|2|1}] + [e0|{2|6|2|4}] + [e0|{2|6|5|3}] + [e0|{2|6|5|2}] + [e0|{2|6|5|1}] + [e0|{2|6|5|4}] + [e0|{2|6|1|7}] + [e0|{2|6|1|3}] + [e0|{2|6|1|2}] + [e0|{2|6|1|5}] + [e0|{2|6|4|7}] + [e0|{2|6|4|2}] + [e0|{2|2|7|3}] + [e0|{2|2|7|6}] + [e0|{2|2|7|2}] + [e0|{2|2|7|5}] + [e0|{2|2|7|1}] + [e0|{2|2|7|4}] + [e0|{2|2|3|6}] + [e0|{2|2|3|2}] + [e0|{2|2|3|5}] + [e0|{2|2|3|1}] + [e0|{2|2|6|3}] + [e0|{2|2|6|2}] + [e0|{2|2|6|4}] + [e0|{2|2|2|7}] + [e0|{2|2|2|3}] + [e0|{2|2|2|6}] + [e0|{2|2|5|3}] + [e0|{2|2|5|6}] + [e0|{2|2|4|7}] + [e0|{2|5|7|3}] + [e0|{2|5|7|6}] + [e0|{2|5|3|5}] + [e0|{2|5|3|1}] + [e0|{2|5|3|4}] + [e0|{2|5|6|7}] + [e0|{2|5|6|3}] + [e0|{2|5|6|2}] + [e0|{2|5|6|1}] + [e0|{2|5|6|4}] + [e0|{2|5|2|6}] + [e0|{2|5|5|6}] + [e0|{2|1|7|3}] + [e0|{2|1|7|6}] + [e0|{2|1|7|4}] + [e0|{2|1|3|7}] + [e0|{2|1|3|6}] + [e0|{2|1|6|3}] + [e0|{2|1|6|2}] + [e0|{2|1|6|5}] + [e0|{2|1|6|1}] + [e0|{2|1|6|4}] + [e0|{2|1|2|6}] + [e0|{2|1|1|7}] + [e0|{2|4|7|3}] + [e0|{2|4|7|2}] + [e0|{2|4|7|5}] + [e0|{2|4|7|1}] + [e0|{2|4|3|7}] + [e0|{2|4|3|6}] + [e0|{2|4|3|4}] + [e0|{2|4|2|7}] + [e0|{5|7|7|3}] + [e0|{5|7|7|6}] + [e0|{5|7|3|7}] + [e0|{5|7|3|1}] + [e0|{5|7|3|4}] + [e0|{5|7|6|3}] + [e0|{5|7|6|2}] + [e0|{5|7|6|5}] + [e0|{5|7|5|3}] + [e0|{5|3|7|2}] + [e0|{5|3|7|5}] + [e0|{5|3|7|1}] + [e0|{5|3|3|7}] + [e0|{5|3|3|6}] + [e0|{5|3|3|2}] + [e0|{5|3|3|1}] + [e0|{5|3|6|6}] + [e0|{5|3|6|5}] + [e0|{5|3|2|7}] + [e0|{5|3|2|6}] + [e0|{5|3|1|4}] + [e0|{5|3|4|3}] + [e0|{5|3|4|1}] + [e0|{5|6|7|7}] + [e0|{5|6|7|2}] + [e0|{5|6|7|5}] + [e0|{5|6|7|1}] + [e0|{5|6|7|4}] + [e0|{5|6|3|6}] + [e0|{5|6|3|2}] + [e0|{5|6|3|5}] + [e0|{5|6|3|1}] + [e0|{5|6|6|7}] + [e0|{5|6|6|2}] + [e0|{5|6|2|7}] + [e0|{5|6|2|3}] + [e0|{5|6|2|6}] + [e0|{5|6|2|2}] + [e0|{5|6|2|5}] + [e0|{5|6|2|4}] + [e0|{5|6|5|2}] + [e0|{5|6|5|1}] + [e0|{5|6|1|3}] + [e0|{5|6|1|6}] + [e0|{5|6|1|2}] + [e0|{5|6|1|5}] + [e0|{5|6|1|4}] + [e0|{5|6|4|7}] + [e0|{5|6|4|3}] + [e0|{5|2|7|3}] + [e0|{5|2|3|3}] + [e0|{5|2|3|6}] + [e0|{5|2|3|5}] + [e0|{5|2|3|1}] + [e0|{5|2|3|4}] + [e0|{5|2|6|7}] + [e0|{5|2|6|3}] + [e0|{5|2|6|2}] + [e0|{5|2|2|3}] + [e0|{5|2|2|6}] + [e0|{5|2|5|3}] + [e0|{5|5|7|3}] + [e0|{5|5|7|6}] + [e0|{5|5|3|7}] + [e0|{5|5|3|3}] + [e0|{5|5|3|2}] + [e0|{5|5|3|1}] + [e0|{5|5|6|6}] + [e0|{5|5|6|2}] + [e0|{5|5|6|5}] + [e0|{5|5|6|1}] + [e0|{5|5|5|6}] + [e0|{1|7|7|7}] + [e0|{1|7|7|2}] + [e0|{1|7|7|4}] + [e0|{1|7|3|7}] + [e0|{1|7|6|7}] + [e0|{1|7|6|3}] + [e0|{1|7|6|5}] + [e0|{1|7|6|1}] + [e0|{1|7|2|7}] + [e0|{1|7|2|5}] + [e0|{1|7|5|5}] + [e0|{1|7|1|7}] + [e0|{1|7|1|6}] + [e0|{1|7|1|2}] + [e0|{1|7|1|5}] + [e0|{1|7|1|1}] + [e0|{1|7|4|7}] + [e0|{1|7|4|3}] + [e0|{1|3|7|6}] + [e0|{1|3|7|2}] + [e0|{1|3|7|5}] + [e0|{1|3|7|1}] + [e0|{1|3|3|7}] + [e0|{1|3|6|3}] + [e0|{1|3|6|5}] + [e0|{1|3|2|6}] + [e0|{1|3|1|7}] + [e0|{1|6|7|3}] + [e0|{1|6|7|5}] + [e0|{1|6|3|6}] + [e0|{1|6|3|2}] + [e0|{1|6|3|5}] + [e0|{1|6|6|3}] + [e0|{1|6|6|5}] + [e0|{1|6|6|4}] + [e0|{1|6|2|3}] + [e0|{1|6|2|2}] + [e0|{1|6|2|5}] + [e0|{1|6|2|4}] + [e0|{1|6|1|2}] + [e0|{1|6|1|5}] + [e0|{1|6|1|4}] + [e0|{1|6|4|3}] + [e0|{1|6|4|2}] + [e0|{1|6|4|5}] + [e0|{1|2|7|3}] + [e0|{1|2|7|6}] + [e0|{1|2|7|5}] + [e0|{1|2|7|1}] + [e0|{1|2|7|4}] + [e0|{1|2|6|3}] + [e0|{1|2|6|2}] + [e0|{1|2|6|4}] + [e0|{1|1|7|5}] + [e0|{1|1|3|7}] + [e0|{1|1|3|6}] + [e0|{1|1|6|1}] + [e0|{1|1|2|7}] + [e0|{4|7|2|3}] + [e0|{4|7|1|5}] + [e0|{4|7|1|1}] + [e0|{4|7|4|5}] + [e0|{4|3|7|6}] + [e0|{4|3|7|5}] + [e0|{4|3|3|7}] + [e0|{4|3|3|2}] + [e0|{4|3|3|4}] + [e0|{4|3|6|7}] + [e0|{4|3|6|6}] + [e0|{4|3|6|5}] + [e0|{4|3|6|1}] + [e0|{4|3|2|2}] + [e0|{4|3|2|5}] + [e0|{4|3|2|4}] + [e0|{4|3|5|3}] + [e0|{4|3|5|2}] + [e0|{4|3|1|5}] + [e0|{4|3|4|2}] + [e0|{4|6|7|3}] + [e0|{4|6|7|1}] + [e0|{4|6|7|4}] + [e0|{4|6|3|7}] + [e0|{4|6|3|3}] + [e0|{4|6|3|2}] + [e0|{4|6|3|5}] + [e0|{4|6|3|4}] + [e0|{4|6|6|7}] + [e0|{4|6|2|7}] + [e0|{4|6|4|3}] + [e0|{4|2|7|2}] + [e0|{4|2|7|5}] + [e0|{4|2|7|1}] + [e0|{4|2|3|7}] + [e0|{4|2|3|3}] + [e0|{4|2|3|6}] + [e0|{4|2|3|4}] + [e0|{4|2|6|7}] + [e0|{4|2|6|3}] + [e0|{4|2|2|3}] + [e0|{4|4|3|3}] + [e0|{4|4|6|7}] + [e0|{4|4|6|3}] + [e12|{7|7|2}] + [e11|{7|7|5}] + [e12|{7|7|1}] + [e12|{7|3|7}] + [e11|{7|3|3}] + [e11|{7|3|4}] + [e11|{7|6|7}] + [e12|{7|6|7}] + [e12|{7|6|3}] + [e12|{7|6|6}] + [e11|{7|6|2}] + [e12|{7|6|5}] + [e12|{7|6|1}] + [e11|{7|6|4}] + [e11|{7|2|3}] + [e11|{7|2|6}] + [e11|{7|2|2}] + [e12|{7|2|2}] + [e11|{7|2|5}] + [e11|{7|2|1}] + [e12|{7|2|4}] + [e11|{7|5|3}] + [e12|{7|5|3}] + [e12|{7|5|6}] + [e11|{7|5|2}] + [e11|{7|1|7}] + [e12|{7|1|2}] + [e12|{7|1|5}] + [e12|{7|1|4}] + [e11|{7|4|3}] + [e11|{7|4|2}] + [e12|{7|4|2}] + [e11|{7|4|1}] + [e11|{7|4|4}] + [e11|{3|7|7}] + [e11|{3|7|3}] + [e11|{3|7|6}] + [e11|{3|7|2}] + [e12|{3|7|2}] + [e12|{3|7|1}] + [e11|{3|7|4}] + [e12|{3|7|4}] + [e11|{3|3|7}] + [e12|{3|3|7}] + [e11|{3|3|2}] + [e12|{3|3|2}] + [e11|{3|3|1}] + [e12|{3|6|7}] + [e12|{3|6|3}] + [e12|{3|6|6}] + [e12|{3|6|5}] + [e12|{3|6|1}] + [e11|{3|2|7}] + [e11|{3|2|6}] + [e12|{3|2|2}] + [e11|{3|2|5}] + [e11|{3|2|1}] + [e11|{3|2|4}] + [e12|{3|2|4}] + [e11|{3|5|7}] + [e12|{3|5|3}] + [e11|{3|5|1}] + [e11|{3|1|7}] + [e11|{3|1|6}] + [e12|{3|1|6}] + [e11|{3|1|2}] + [e11|{3|1|5}] + [e12|{3|1|5}] + [e11|{3|4|7}] + [e11|{3|4|3}] + [e11|{3|4|6}] + [e11|{3|4|2}] + [e12|{3|4|2}] + [e11|{3|4|1}] + [e12|{6|7|7}] + [e11|{6|7|3}] + [e12|{6|7|2}] + [e12|{6|7|5}] + [e11|{6|3|7}] + [e12|{6|3|6}] + [e11|{6|3|5}] + [e11|{6|3|4}] + [e12|{6|3|4}] + [e12|{6|6|7}] + [e12|{6|6|3}] + [e11|{6|6|2}] + [e11|{6|6|5}] + [e11|{6|2|7}] + [e12|{6|2|7}] + [e11|{6|2|3}] + [e11|{6|2|2}] + [e11|{6|2|5}] + [e12|{6|2|1}] + [e11|{6|5|7}] + [e12|{6|5|7}] + [e12|{6|5|3}] + [e11|{6|5|2}] + [e12|{6|5|2}] + [e11|{6|5|5}] + [e12|{6|5|1}] + [e11|{6|1|7}] + [e12|{6|1|7}] + [e12|{6|1|6}] + [e11|{6|1|2}] + [e11|{6|1|5}] + [e12|{6|1|5}] + [e11|{6|1|1}] + [e12|{6|1|1}] + [e11|{6|1|4}] + [e12|{6|1|4}] + [e12|{6|4|7}] + [e12|{6|4|3}] + [e11|{6|4|5}] + [e11|{6|4|1}] + [e12|{2|7|7}] + [e11|{2|7|3}] + [e11|{2|7|1}] + [e12|{2|7|1}] + [e11|{2|7|4}] + [e12|{2|3|7}] + [e11|{2|3|3}] + [e12|{2|3|3}] + [e11|{2|3|6}] + [e11|{2|3|5}] + [e11|{2|6|7}] + [e11|{2|6|3}] + [e11|{2|6|6}] + [e12|{2|6|5}] + [e11|{2|6|1}] + [e12|{2|6|1}] + [e11|{2|2|7}] + [e12|{2|2|7}] + [e12|{2|2|3}] + [e11|{2|2|6}] + [e11|{2|5|7}] + [e11|{2|1|3}] + [e11|{5|7|7}] + [e11|{5|7|3}] + [e11|{5|3|7}] + [e11|{5|3|2}] + [e11|{5|3|4}] + [e11|{5|6|7}] + [e11|{5|6|3}] + [e11|{5|6|2}] + [e11|{5|2|7}] + [e11|{5|2|3}] + [e11|{5|2|6}] + [e11|{5|5|3}] + [e11|{1|7|7}] + [e11|{1|7|6}] + [e11|{1|7|2}] + [e11|{1|3|7}] + [e11|{1|3|3}] + [e11|{1|6|2}] + [e11|{1|2|7}] + [e11|{1|2|3}] + [e11|{1|2|6}] + [e11|{1|1|7}] + [e12|{4|7|1}] + [e12|{4|3|7}] + [e12|{4|6|7}] + [e12|{4|6|3}] + [e21|{7|7}] + [e21|{7|6}] + [e21|{7|2}] + [e21|{7|1}] + [e22|{7|1}] + [e22|{7|4}] + [e22|{3|3}] + [e21|{3|6}] + [e22|{3|2}] + [e21|{3|5}] + [e22|{3|5}] + [e22|{3|4}] + [e22|{6|7}] + [e21|{6|3}] + [e22|{6|3}] + [e21|{6|5}] + [e21|{2|7}] + [e22|{2|3}] + [e21|{2|2}] + [e22|{2|2}] + [e3|{7}] + [e3|{3}]
$, 
\else
The outputs of the program indicates a solution $u'$ of (\ref{eq:Equation 2}) which is given as follows:
\begin{align*}
u'_{*} = &
[e^0|\{a^3b|a^3b|a^3b|a^{2}b\}]^{*} + [e^0|\{a^3b|a^3b|a^2b|a^2b\}]^{*} + [e^0|\{a^3b|a^3b|a^{2}b|a^2\}]^{*} 
\\&
+ [e^0|\{a^3b|a^3b|a^2|ab\}]^{*} + \cdots + [e^3|\{a^3b\}]^{*} + [e^{3}|\{a^{3}\}]^{*},
\end{align*}
\fi
where $u'_{*}$ denotes the dual of $u'$ in $C_{4}(S^3\times_{\ad{}}P^{4}G;\field_{2})$.

Thus (\ref{eq:Equation 1}) has one solution, and it completes the proof of Lemma \ref{lem:mainlem}. 

\begin{remark}
If we consider (\ref{eq:Equation 2}) in $C^{5}(S^3\times_{\ad{}}P^{5}G;\field_{2}) = C^{5}(S^3\times_{\ad{}}P^{\infty}G;\field_{2})$ instead of $C^{5}(S^3\times_{\ad{}}P^{4}G;\field_{2})$, we have no solution as the python program produces the following outputs.
\par\smallskip
{\rm\begin{lstlisting}[basicstyle=\ttfamily\footnotesize, frame=single, breaklines=true]
The Number of 4-cells is 3192.
The Number of 5-cells is 22344.
The size of the coefficients matrix delta is 22344 * 3192.
The size of the augmented coefficients matrix Delta is 22344 * 3193.
The rank of the matrix delta is 2789.
The rank of the matrix Delta is 2790.
\end{lstlisting}}
It should imply that $z \otimes x^{2}$ is non-zero in $H^{*}(S^3\times_{\ad{}}P^{5}G;\field_{2})$ and $\wgtB{z \otimes x^{2};\field_{2}}=5$.
\end{remark}

\ifdefined\reviseone
\section*{Acknowledgements}
The authors are very grateful to Takeshi Nanri for his kind advice to improve python program.
The computation was carried out using the computer resource offered under the category of Trial Use Projects by Research Institute for Information Technology (R{\hskip.075em}I\raise.75ex\hbox{\small$2$}{\hskip.1em}T), Kyushu University.
In fact, we obtained a direct solution for ${}^{t}T_{\!\delta}\,\mathbold{x} = {}^{t}T_{\!c}$, while it is too big to describe explicitly here.
\par\smallskip
\begin{lstlisting}[basicstyle=\ttfamily\footnotesize, frame=single, breaklines=true]
The Number of 5-cells is 22344.
The Number of 6-cells is 38759.
The size of the coefficients matrix delta is 38759x22344.
The size of the augmented coefficients matrix Delta is 38759x22345.
The rank of the matrix delta is 15724.
The rank of the matrix Delta is 15724.
The length of one particular solution is 5546.
\end{lstlisting}\vskip0ex
\par
The first named author is supported in part by Grant-in-Aid for Scientific Research (S) \#17H06128, by Exploratory Research \#18K18713 and by Grant-in-Aid for Scientific Research (C) \#23K03093 from Japan Society for the Promotion of Science.
\fi
%
%
\bibliographystyle{alpha}
\bibliography{topcomp}
\end{document}